\documentclass[11pt]{amsart}
\usepackage{amsmath,amsthm,amscd,amsfonts,amssymb,graphicx,color,MnSymbol,enumerate,tikz-cd}
\usepackage{hyperref,cleveref}
\usepackage{xcolor}
\usepackage[labelformat=empty]{caption,subcaption}
\usepackage{inputenc}
\usepackage{stmaryrd}

\newtheorem{theorem}{Theorem}[section]
\newtheorem{lemma}[theorem]{Lemma}
\newtheorem{proposition}[theorem]{Proposition}
\newtheorem{corollary}[theorem]{Corollary}
\theoremstyle{definition}
\newtheorem{definition}[theorem]{Definition}

\newtheorem{conjecture}[theorem]{Conjecture}

\theoremstyle{remark}
\newtheorem{remark}[theorem]{Remark}

\numberwithin{equation}{section}

\allowdisplaybreaks

\begin{document}
	
	\title[On simple extensions generated by Key Polynomials]{ On simple extensions generated by Key Polynomials}
	\author[Nikita Dwivedi]{Nikita Dwivedi}
	\address{Department of Mathematics\\ University of Delhi\\  Delhi-110007, India.}
	\email{ndwivedi@maths.du.ac.in}
	\author[Anuj Bishnoi]{Anuj Bishnoi$^\ast$}
	\address{Department of Mathematics\\  University of Delhi \\   Delhi-110007, India.}
	\email{abishnoi@maths.du.ac.in}
	
	\begin{abstract}
Let $K^h$ be a Henselization of a valued field $(K,v),$ $w$ a valuation-transcendental extension of $v$ to $K[x],$ and $\phi$ a key polynomial for $w$. In this paper, we prove that if $\phi$ is irreducible over $K^h,$ then the simple extension $L|K,$ generated by a root of any key polynomial for $w,$ is unibranched and its defect is independent of the choice of key polynomial. As a consequence, we generalize some well-known results for Henselian valued fields to arbitrary valued fields. Moreover, we provide some criteria for $L|K$ to be unibranched and defectless in terms of Mac Lane-Vaqui\'e chains, complete sequences of abstract key polynomials and saturated distinguished chains. As an application, we generalize a classical result of Ore about the existence of $p$-regular generators for number fields.
	\end{abstract}
	\thanks{$^\ast$Corresponding author, E-mail address: abishnoi@maths.du.ac.in}
	\maketitle
	
	
	\section{Introduction}
Let $(K,v)$ be a valued field, $\bar{v}$ be an extension of $v$ to the fixed algebraic closure $\overline{K}$ of $K,$ and $K^h$ be the Henselization of $K$ with respect to $\bar{v}.$ Let $w$ be a valuation-transcendental extension of $v$ to $K[x],$ and $\overline{w}$ be a common extension of $w$ and $\bar{v}$ to $\overline{K}[x].$ In 2021, Mahboub et al. gave an upper bound on the number of common extensions $\overline{w}$ (\cite[Corollary 5.7]{W}). In 2025, Nart generalized this upper bound and gave the description of common extensions in terms of key polynomials for $w$ (see Theorem \ref{t2}). In Section 3, we provide an explicit formula for the number of common extensions $\overline{w}$ of $w$ and $\bar{v},$ in terms of the $j$-invariant (see Theorem \ref{p23}). Using this, we obtain a relation among the different key polynomials for $w$, including their irreducibility and degrees over $ K^h.$ In particular, we prove that if any key polynomial for $w$ is irreducible over $K^h,$ then the simple extension generated by a root of any key polynomial for $w$ is unibranched (Proposition \ref{bp2}).
	
	In 2023, Nart and Novacoski introduced the defect of a valuation on $K[x]$ and proved that the defect of a valuation with nontrivial support is equal to a Henselian defect of the corresponding simple extension \cite[Theorem 6.14]{EN3}. Given any valuation-transcendental extension $w$ and its minimal degree key polynomial $\phi,$ we can always construct a valuation $v_\phi$ with nontrivial support $\phi K[x].$ Also from the Mac Lane-Vaqui\'e chain of $w,$ we can obtain the Mac Lane-Vaqui\'e chain of $v_\phi.$ Keeping this in mind, we prove in Section 4 that the defect of $w|v$ is equal to the defect of $v_\phi|v.$ Therefore, we have that the defect of $w|v$ is also equal to a Henselian defect of the simple extension generated by a root of $\phi$ (see Theorem \ref{odt2}), which in turn extends the result of Nart and Novacoski to valuation-transcendental extensions. As a consequence, we obtain that if $\phi$ is irreducible over $K^h,$ then the defect of a simple extension, generated by a root of any key polynomial for $w,$ is an invariant for $w$ (Theorem \ref{2.12}).
	
	Distinguished pairs were first introduced by Popescu and Zaharescu in \cite{PZ} for local fields and later generalized to arbitrary valued fields. The characterization of distinguished pairs in terms of key polynomials was given by Aghigh in \cite{AN1} for Henselian valued fields. In Section 5, we generalize this characterization for an arbitrary valued field (see Theorem \ref{3.9}). For a Henselian valued field, Aghigh et al. also proved several properties of distinguished pairs and distinguished chains in \cite{AK2}, \cite{AN1} and \cite{AN2}. We generalize these properties to arbitrary valued fields, assuming that the simple extensions generated by a distinguished pair are unibranched.
	
	The notion of abstract key polynomials was introduced in \cite{JD} and \cite{NS} as an alternative to key polynomials. The relation between key polynomials and abstract key polynomials is also given in \cite{JD} and \cite{Ma}. In 2023, Mavi and Bishnoi \cite{SA3}, gave a characterization between abstract key polynomials and distinguished pairs for Henselian valued fields. In Section 6, we generalize these characterization to an arbitrary valued field assuming that one of the abstract key polynomials in the pair is irreducible over $K^h$ (see Lemma \ref{abl1} and Theorem \ref{abth}). As a consequence, we also generalize some results relating the complete sequence of abstract key polynomials and saturated distinguished chains to arbitrary valued fields.
	
	Finally, using the results of this paper, we prove the following theorem, which gives a characterization relating the Mac Lane-Vaqui\'e chain, the complete sequence of abstract key polynomials, and the saturated distinguished chain for almost all valuations on $K[x].$
	
	 \begin{theorem}\label{th1}
		Let $w$ be any valuation on $K[x]$ and
		\begin{align*}
			(v\xrightarrow{\phi_0,\gamma_0})	w_0\xrightarrow{\phi_1,\gamma_1} w_1\xrightarrow{\phi_2,\gamma_2}\cdots \longrightarrow w_{n-1} \xrightarrow{\phi_{n},\gamma_{n}} w_{n}\longrightarrow\cdots\longrightarrow w
		\end{align*}
		be the Mac Lane-Vaqui\'e chain of $w.$ Then for every $n\in \mathbb{N},$ the following are equivalent:
		\begin{enumerate}[(i)]
			\item The augmentation step $w_i\longrightarrow w_{i+1}$ is ordinary, for every $i,\ 0\leq i\leq n-1$.
			\item The valuation $w_n$ has  a finite complete sequence $\{\phi_0, \phi_{1},\ldots,\phi_n\}$   of abstract key polynomials.
			\item  The polynomial $\phi_n$  has a saturated distinguished chain $(\phi_n, \phi_{n-1},\ldots, \phi_0).$ 
			\item $\theta_n$ has a saturated distinguished chain $\theta_n, \theta_{n-1},\ldots, \theta_0,$ where $\theta_i$ is an optimizing root of $\phi_i,$ for every $i,\ 0\leq i\leq n.$
			\item $K(\theta_n)|K$ is unibranched and defectless, for every $\theta_n\in Z(\phi_n).$
		\end{enumerate}
	\end{theorem}
	 Note that in the above theorem, (iv)$\iff$(v) generalizes Theorem 1.2 of \cite{AK2} to an arbitrary valued field. However, the same result is also proved in \cite[Theorem 1.2]{AR} with a different approach. Also (ii)$\iff$(v) generalizes Theorem 5.9 of \cite{SD}. The above theorem can also be interpreted in the language of Okutsu frame of $\phi_n$ (see \cite[Theorem 4.4 and Theorem 4.5]{M2}).
	 
	 In view of the above theorem, we have that if $K(\theta_n)|K$ is unibranched and defectless, then the length of an MLV chain of $v_{\phi_n}$, i.e., $\text{depth}(\theta_n)$ is equal to the number of terms in the distinguished chain of $\theta_n$ over $K.$ Hence, as an immediate consequence of Theorem \ref{dpth3}, we prove the following Conjecture of Nart and Novacoski \cite{NN}, in more general setting (see Corollary \ref{lc}).
	 \begin{conjecture}\label{1.2}
	 	For a henselian valued field $(K,v)$, let $L|K$ be a finite simple defectless extension, and $s, t$ be the minimal number of generators of $k_L|k_v$ and $\Gamma_L|\Gamma_v$ respectively. Then, \[\max\{s,t\}\leq \textup{depth}(L|K,v)\leq s+t.\]
	 \end{conjecture}
	 This conjecture also generalizes a classical result of Ore about the existence of $p$-regular generators, to defectless extensions of henselian valued fields.
	
	\section{Preliminaries}
		 In this section, we recall some notation, definitions, and results that will be used in the proofs of our main results.
		
 A surjective map $v:K \longrightarrow\ \Gamma_v\cup\{\infty\},$ where $\Gamma_v$ is a totally ordered additively written abelian group, is called a {\bf valuation}, if it satisfies the following axioms, for all $a,b$ in $K$:
	\begin{enumerate}[(i)]
		
\item $v(ab)=v(a)+v(b)$
	\item $v(a+b)\geq \min\{v(a),v(b)\}$
			\item $v(a)=\infty$ if and only if $a=0$.
	\end{enumerate}
		The pair $(K,v)$ is called a {\bf valued field} of arbitrary rank, and  $\Gamma_ v$  its {\bf value group}. The set $O_v=\{a\in K \mid v(a)\geq 0\}$ is a subring of $K,$ called the {\bf valuation ring} which has a unique {\bf maximal ideal} $M_v=\{a\in K\mid v(a)>0 \}.$ The quotient $O_v/M_v$ is called the {\bf residue field} of $v$ and is denoted by $k_v.$  For any $a\in O_v,$ the $v$-residue of $a$, denoted by $a^*,$ is the image of $a$ under the canonical homomorphism from $O_v$ onto $k_v.$ Let $\bar{v}$ be an extension of $v$ to a fixed algebraic closure $\overline{K}$ of $K$ with value group $\Gamma_{\bar{v}}.$ 
		
 Let $w:K[x]\longrightarrow\Gamma\cup\{\infty\}$ be an extension of $v$ satisfying axioms (i) and (ii) above, where $\Gamma$ is some ordered abelian group containing $\Gamma_{\bar{v}}.$ The {\bf support}  of $w$ is the prime ideal
 \[\text{supp}(w)=w^{-1}(\infty).\] If $\text{supp}(w)=\{0\},$ then it extends uniquely to a valuation on $K(x).$ If $\text{supp}(w)\neq\{0\},$ then we say that $w$ has a nontrivial support. In either case, we call $w$ a valuation on $K[x],$ and define the value group $\Gamma_w$ of $w$ as the group generated by $w(K[x]\backslash\text{supp}(w)),$ and the residue field $k_w$ of $w$ as the residue field of the canonical valuation induced by $w$ (see Remark \ref{r21}) on the quotient field of $K[x]/\text{supp}(w)$.

 An extension $w$ of $v$ to $K(x)$ satisfies the well-known Abhyankar inequality: $\text{rr}(\Gamma_w / \Gamma_v) + \text{tr.deg.}[k_{w}: k_v] \leq 1,$ where $\text{rr}(\Gamma_w / \Gamma_v )$ is the rational rank of $\Gamma_ w / \Gamma_ v,$ and $\text{tr.deg.}[k_{w}: k_v]$ is the transcendence degree of $k_{w}$ over $k_v.$ The extension $w$ is said to be {\bf value-transcendental} if $\text{rr}(\Gamma_w / \Gamma_v)=1$ and is said to be {\bf residue-transcendental} if $\text{tr.deg.}[k_{w}: k_v]=1$. We call $w$ {\bf valuation-transcendental} if it is either value-transcendental or residue-transcendental. Otherwise, $w$ is called {\bf valuation-algebraic}.

An extension $\overline{w}$ of $w$ to $\overline{K}[x]$ which is also an extension of $\bar{v}$ is called a {\bf common extension} of $w$ and $\bar{v}$. 

For any pair $(\alpha,\delta)\in\overline{K}\times\Gamma\cup\{\infty\},$ the map $\overline{w}_{\alpha,\delta}: \overline{K}[x]\longrightarrow \Gamma\cup\{\infty\},$ given by
	\begin{align*}
		\overline{w}_{\alpha,\delta}\left(\sum_{i\geq 0} c_i (x-\alpha)^i\right):=\min_{i\geq 0}\{\bar{v}(c_i)+i\delta\}, \, c_i\in\overline{K},
	\end{align*}
	is a valuation on $\overline{K}[x]$ and is said to be defined by $\min,\, \bar{v},\, \alpha$ and $\delta.$ If $\delta=\infty,$ then it has a nontrivial support generated by $(x-\alpha).$
		Let $\overline{w}$ be a common extension of $w$ and $\bar{v}$ to $\overline{K}[x]$ such that $\overline{w}=\overline{w}_{\alpha,\delta},$ then $(\alpha,\delta)$ is called a {\bf pair of definition} for $\overline{w}.$ We denote its restriction $w$ to $K[x]$ by $w_{\alpha,\delta}.$
	\begin{lemma}(\cite[Lemma 2.4]{W})\label{il1}
		Let $(\alpha,\delta)$ be a pair of definition for $\overline{w}$ and $(\alpha',\delta')\in \overline{K}\times\Gamma.$ Then $(\alpha',\delta')$ is also a pair of definition for $\overline{w}$ if and only if $\delta'=\delta$ and $\bar{v}(\alpha-\alpha')\geq \delta.$
	\end{lemma}
	
	\begin{definition}
	A pair $(\alpha,\delta)$ in $\overline{K}\times \Gamma$ is called  a $(K,v)$-{\bf minimal pair} if for every $\beta$ in $\overline{K}$ satisfying $\bar{v}(\alpha-\beta)\geq\delta,$ we have  $\deg\beta\geq\deg\alpha,$ where by $\deg\alpha$ we mean the degree of the extension $K(\alpha)|K.$  
\end{definition}

 If $\overline{w}=\overline{w}_{\alpha,\delta},$ such that $(\alpha,\delta)$ is a $(K,v)$-minimal pair, then we say that $(\alpha,\delta)$ is a {\bf minimal pair of definition} for $\overline{w}.$ In view of Lemma \ref{il1}, any pair of definition for $\overline{w}$ can be replaced by a minimal pair of definition.

 For any pair $(\alpha,\delta)\in\overline{K}\times\Gamma,$ we define the open and closed balls with center $\alpha$ and radius $\delta$ as\[B^\circ(\alpha,\delta)=\{\beta\in\overline{K}\mid\bar{v}(\beta-\alpha)> \delta\}\subseteq B(\alpha,\delta)=\{\beta\in\overline{K}\mid\bar{v}(\beta-\alpha)\geq \delta\}.\]
For any $(\alpha,\delta), (\alpha',\delta')\in \overline{K}\times\Gamma,$ we have \[\overline{w}_{\alpha,\delta}(f)\leq \overline{w}_{\alpha',\delta'}(f),\ \text{for all} \ f\in\overline{K}[x]\iff \delta\leq \delta' \ \text{and} \ B(\alpha',\delta')\subseteq B(\alpha,\delta).\]

If  $\delta\in\Gamma,$ then $w_{\alpha,\delta}$ is a valuation-transcendental extension, and any valuation-transcendental extension can be obtained in this way \cite[Corollary 3.7]{W}. The valuation $w$ is residue-transcendental if $\delta\in\Gamma_{\overline{v}}$ and value-transcendental if $\delta\notin\Gamma_{\overline{v}}$. On the other hand, if $\delta=\infty,$ then $w_{\alpha,\delta}$ has a nontrivial support generated by the minimal polynomial of $\alpha$ over $K.$ If $\alpha\in K,$ then the valuation $w_{\alpha,\delta}$ is called a {\bf depth zero valuation}.

\subsection{Key Polynomials}

We now recall the definition of key polynomials, first introduced by Mac Lane \cite{M} in 1936 for discrete rank-one valuations and later generalized by Vaqui\'e \cite{V} in 2007 for arbitrary valuations. Further, Nart \cite{EN1} classified all possible extensions of $v$ to $K[x]$ with the help of key polynomials.

For a valuation-transcendental extension  $w$ of $v$ to $K(x)$ and polynomials $f,$ $g$ in $K[x],$ we say that $f$ and $g$ are {\bf $w$-equivalent} (denoted $f\sim_w g$)  if $w(f-g)>w(f)=w(g)$ and $g$ is {\bf $w$-divisible} by $f$ (denoted $f\mid_{w} g$)  if there exists some polynomial $h \in K[x]$ such that $g$ is $w$-equivalent to $fh.$

		\begin{definition}
		A monic polynomial $f$   is called a  {\bf  key polynomial} for $w$ if it is
		\begin{enumerate}[(i)]
			\item  $w$-{\bf irreducible}, i.e., for any $g,\, h\in K[x],$ whenever $f\mid_{w} gh,$ then either $f\mid_{w}g $ or $f\mid_{w}h,$ and 
			\item $w$-{\bf minimal}, i.e., for every nonzero polynomial $g\in K[x],$ whenever $f\mid_{w}g,$ then $\deg g\geq \deg f.$
		\end{enumerate}
	\end{definition}
	
	The set of all key polynomials for $w$ is denoted by $\operatorname{KP}(w).$ If $w=w_{\alpha,\delta}$ for some $(K,v)$-minimal pair $(\alpha,\delta)\in \overline{K}\times\Gamma$ and $\phi$  is the minimal polynomial of $\alpha$ over $K,$ then $\phi$ is a minimal degree key polynomial for $w$ (\cite[Theorem 1.1]{JN} and \cite[Theorem 2.21]{Ma}). 
The existence of key polynomials is characterized as follows.
\begin{theorem}(\cite[Theorem 4.4 and 4.2]{EN4})
	A valuation $w$ on $K(x)$ has $\operatorname{KP}(w)\neq\emptyset$ if and only if it is valuation-transcendental. Moreover, if $w$ is value-transcendental, then all key polynomials have the same degree.
\end{theorem}
If $\operatorname{KP}(w)\neq\emptyset,$ then we define $\deg (w)=\min_{f\in\operatorname{KP}(w)}\deg f.$ If $\text{supp}(w)=\phi K[x]$ for some $\phi\in K[x],$ then we define $\deg (w)=\deg\phi.$

		Let $\phi\in K[x]$ be  $w$-minimal, then by Proposition 2.3 of \cite{EN4}, for  any polynomial $f\in K[x]$ with $\phi$-expansion $\sum_{i=0}^{n}   a_i \phi^i,$ $a_i\in K[x],$  $\deg a_i<\deg \phi,$ we have 
	\begin{align*}
		w(f)=\min_{0 \leq i \leq n}\{w(a_i)+ iw(\phi)\}.
	\end{align*}
	Consider the set $S_{w,\phi}(f):=\{0\leq i\leq n\mid w(f)=w(f_i\phi^i)\}.$
	If $\phi$ has minimal degree in $\operatorname{KP}(w),$ then for any nonzero $f\in K[x],$ we denote   
	\[\deg_w(f):=\max(S_{w,\phi}(f)).\] Clearly, it is independent of the choice of $\phi$ among all minimal degree key polynomials for $w.$
	\begin{lemma}(\cite[Theorem 3.7]{EN4})\label{kpl1}
		For a valuation-transcendental extension $w$, any nonzero $f\in K[x]$ is $w$-minimal if and only if $\deg f=\deg_w(f)\deg(w).$	
	\end{lemma}
	
	For  valuations $w$ and $w'$ on $K[x],$ taking values in the  common ordered abelian group $\Gamma,$ we say that $w\leq w'$ if and only if 
	$$w(f)\leq w'(f),~\forall f\in K[x].$$
	\begin{definition}
		Let $w<w',$ then the set of all monic polynomials $g\in K[x]$ of minimal degree (say) $d'$ such that $w(g)<w'(g),$ denoted  by $\Phi(w,w')$ is called the {\bf tangent direction} of $w,$ and  $\deg(\Phi(w,w')) :=d'.$ 
\end{definition}
 Note that the set $\Phi(w,w')\subset \operatorname{KP}(w)$ (see \cite[Theorem 1.15]{V}).
\subsection{Graded Algebra}
	
	We now define the graded algebra associated with the valuation $w$ on $K[x]$ (\cite{EN4} and \cite{EN3}). For all $\alpha\in \Gamma_w,$ consider the abelian groups
\[\mathcal{P}_\alpha:=\{f\in K[x]\mid w(f)\geq \alpha\}\  \text{and} \ \mathcal{P}_\alpha^+:=\{f\in K[x]\mid w(f)> \alpha\}.\]
The \textbf{graded algebra} of $w$ is the integral domain given by:
\[\mathcal{G}_w:= \bigoplus_{\alpha \in \Gamma_w} \mathcal{P}_\alpha  / \mathcal{P}_\alpha^+ .\]
For $f\in K[x],$ we denote its image in $\mathcal{P}_\alpha  / \mathcal{P}_\alpha^+ \subset \mathcal{G}_w $ by $\text{in}_w(f),$ and $\alpha=w(f)\in\Gamma_w$ is called the natural \textbf{grade}, which indicates that $\text{in}_w(f)$ belongs to the homogeneous component $\mathcal{P}_\alpha  / \mathcal{P}_\alpha^+ $.
Let $\mathcal{G}_w^\circ \subset \mathcal{G}_w$ denote the subalgebra generated by the set of all homogeneous units, and $\Gamma_{w} ^{\circ}\subset  \Gamma_{w} $ the subgroup of grades of all homogeneous units in $\mathcal{G}_w.$ Since the homogeneous units are algebraic over the graded algebra of $v$ (\cite[Proposition 3.5]{EN4}), the group $\Gamma_{w} ^{\circ}/\Gamma_v$ is torsion. The {\bf relative ramification index} of $w$ is defined as 
\[e(w):=( \Gamma_{w}:\Gamma_{w} ^{\circ}).\]
 In particular, if $w$ is valuation-transcendental and $\phi\in\operatorname{KP}(w)$ of minimal degree, then
\[\Gamma_{w} ^{\circ}=\{w(a)\mid a(\neq 0)\in K[x],~\deg a<\deg\phi\},\]
and $\Gamma_w=\langle\Gamma_{w} ^{\circ}, w(\phi)\rangle$ \cite[Lemma 2.11]{EN4}.


%

\subsection{Mac Lane-Vaqui\'e chains}

We now define augmentations and Mac Lane-Vaqui\'e (MLV) chains of valuation $w$ on $K[x].$
\begin{definition}
	Let $\phi$ be a key polynomial for a valuation $w$ and $\gamma\in\Gamma\cup\{\infty\}$ such that $\gamma> w(\phi)$. The map $w': K[x]\longrightarrow \Gamma\cup\{\infty\}$ defined by 
	$$w'(f):=\min_{i\geq 0}\{w(f_i)+i\gamma\},$$ where
	$\sum_{i\geq 0}f_i \phi^i, $ $\deg f_i<\deg \phi,$ is the $\phi$-expansion of $f\in K[x],$  gives a valuation on $K[x]$  called the {\bf ordinary augmentation of} $w,$ defined by $\phi$ and $\gamma,$ and is  denoted  by $[w; \phi,\gamma].$
\end{definition}
Note that $w'(\phi)=\gamma,$ i.e., $w(\phi)<w'(\phi).$ If $\gamma<\infty,$ then $\phi\in\operatorname{KP}(w')$ of minimal degree (\cite[Proposition 2.1]{EN1}), otherwise, $\text{supp}(w')=\phi K[x].$ In either case, $\deg(\Phi(w,w'))=\deg (w')=\deg\phi.$ 

Consider a family $\mathcal{W}=(\rho_i)_{i\in\mathbf{A}}$ of valuations on $K(x),$  taking values in a common ordered abelian group $\Gamma,$   and indexed by a totally ordered set $\mathbf{A}.$ 
A polynomial $f$ in $K[X]$ is said to be {\bf $\mathcal{W}$-stable}  if there exists some index $i_0\in\mathbf{A}$ such that
$$\rho_i(f)=\rho_{i_0}(f),~ \forall~ i\geq i_0.$$  This stable value is denoted by $\rho_\mathcal{W}(f).$
%
A polynomial $f\in K[X]$  is {\bf $\mathcal{W}$-unstable} if and only if  
$$\rho_i(f)<\rho_j(f),\hspace{5pt} \forall~ i<j\in\mathbf{A}.$$  
We denote 
$$m_{\infty}:=\min\{\deg f\mid f\in K[X],~\text{$f$ is $\mathcal{W}$-unstable}\}.$$ If  all polynomials are $\mathcal{W}$-stable, then we set $m_{\infty}=\infty$. 

\begin{definition}\label{1.1.14}
	Let $w$ be a valuation on $K(x)$ admitting key polynomials. Then a {\bf continuous family of augmentations} of $w$ is a family of ordinary augmentations of $w$  $$\mathcal{W} =(\rho_i=[w;\chi_i,\gamma_i])_{i\in	\mathbf{A}},$$ indexed by a totally ordered set $\mathbf{A}$ such that $\gamma_i<\gamma_j$ for all $i<j$ in $\mathbf{A},$ satisfying the following conditions:
	\begin{enumerate}[(i)]
		\item The set $\mathbf{A}$  has no last element.
		\item For all $i$ in $\mathbf{A},$ $\chi_i\in \operatorname{KP}(w)$ have the same degree.
		\item For all $i<j$ in $\mathbf{A},$ $\chi_j$ is a  key polynomial for $\rho_i,$ 
		$\chi_j\not\sim_{\rho_i}\chi_i~\text{and}~ \rho_j=[\rho_i;\chi_j,\gamma_j].$
	\end{enumerate}
\end{definition}
The common degree $\deg\chi_i$ for all $i,$ is called the {\bf stable degree} of the family $\mathcal{W}$ and is denoted by $\deg (\mathcal{W}).$

Any continuous family $\mathcal{W}$ of augmentations of $w$ falls in one of the following three cases:
\begin{enumerate}[(i)]
	\item  It has a stable limit,  i.e., $\rho_{\mathcal{W}}$ is a valuation on $K[x],$ if $m_{\infty}=\infty.$ 
	\item It is {\bf in-essential} if $m_{\infty}=\deg(\mathcal{W}).$
	\item It is  {\bf essential }  if $\deg(\mathcal{W})<m_{\infty}<\infty.$
\end{enumerate} 
\begin{definition}
	Let $\mathcal{W}$ be an essential continuous family of augmentations of a valuation $w.$ Then a  monic $\mathcal{W}$-unstable polynomial of minimal degree  is called a {\bf Mac Lane-Vaqui\'e  limit key polynomial} for $\mathcal{W}.$
\end{definition}
%
%
%
%
%
%
\begin{definition}
	Let   $\phi$ be any MLV limit key polynomial for an essential continuous family $\mathcal{W}  =(\rho_i)_{i\in\mathbf{A}}$ of augmentations of $w$   and  $\gamma\in \Gamma\cup\{\infty\}$ such that $\gamma>\rho_i(\phi)$ for all $i\in\mathbf{A}.$ Then the map $w': K[x]\longrightarrow\Gamma\cup\{\infty\}$ defined by
	$$w'(f):=\min_{i\geq 0}\{\rho_\mathcal{W}(f_i)+i\gamma\},$$
	where  $\sum_{i\geq 0} f_i\phi^i,$ $\deg f_i<\deg \phi,$ is the $\phi$-expansion of $f\in K[x],$ gives a valuation on $K[x]$ and is called  the {\bf limit augmentation} of $\mathcal{W},$  denoted by $[\mathcal{W}=(\rho_i)_{i\in\mathbf{A}}; \phi, \gamma].$ 
\end{definition}
Note that $w'(\phi)=\gamma$ and $\rho_i<w'$ for all $i\in\mathbf{A}.$ Also if $\gamma<\infty$, then $\phi$ is a key polynomial for $w'$ of minimal degree  \cite[Proposition 3.5]{EN1}, otherwise, $\text{supp}(w')=\phi K[x].$ In either case, $\deg (w')=\deg\phi.$

\vspace{.20pt}

In 2021, Nart \cite{EN1} introduced the Mac Lane-Vaqui\'e chain for an arbitrary valuation $w$ on $K[x],$ constructed as a mixture of ordinary and limit augmentations satisfying some technical conditions.

\begin{definition}
	A finite, or countably infinite, chain of mixed augmentations 
	\begin{align*}
			(v\xrightarrow{\phi_0,\gamma_0})w_0\xrightarrow{\phi_1,\gamma_1} w_1\xrightarrow{\phi_2,\gamma_2}\cdots \longrightarrow w_{i}\xrightarrow{\phi_{i+1},\gamma_{i+1}} w_{i+1}\longrightarrow\cdots
	\end{align*} is called a {\bf Mac Lane-Vaqui\'e  chain}, if  $w_0$ is a depth zero valuation defined by $(\alpha_0,$ $\gamma_0),$ where $\phi_0=x-\alpha_0,$ and every augmentation step satisfies:
	\begin{enumerate}[(i)]
		\item if $w_{i}\rightarrow w_{i+1}$ is ordinary, then $\deg (w_{i})<\deg(\Phi(w_{i},w_{i+1}))(=\deg(w_{i+1})).$
		\item if $w_{i}\rightarrow w_{i+1}$ is limit, then $\deg (w_{i})=\deg(\Phi(w_{i},w_{i+1}))(<\deg(w_{i+1}))$ and $\phi_{i}\notin\Phi(w_{i},w_{i+1}).$
	\end{enumerate}
	
\end{definition}
The following result categorizes MLV chains of extensions of $v$ to $K[x].$
\begin{theorem}(\cite[Theorem 4.3]{EN1})\label{mlvt1}
	Every valuation $w$ on $K[x]$ falls in one of the following cases. 
	\begin{enumerate}[(i)]
		\item It is the last valuation of a finite MLV chain
		\begin{align*}
			w_0\xrightarrow{\phi_1,\gamma_1} w_1\xrightarrow{\phi_2,\gamma_2}\cdots \longrightarrow w_{n-1}\xrightarrow{\phi_n,\gamma_n} w_n=w.
		\end{align*}
		\item There exists a valuation $w_n$ falling in case (i) such that $w$  is the stable limit of a continuous family ${\mathcal{W}_n}$ of augmentations of $w_n$
		\begin{align*}
			w_0\xrightarrow{\phi_1,\gamma_1} w_1\xrightarrow{\phi_2,\gamma_2}\cdots \longrightarrow w_{n-1}\xrightarrow{\phi_n,\gamma_n} w_n\xrightarrow {{\mathcal{W}_n}} \rho_{\mathcal{W}_n}=w,
		\end{align*}
		such that $\deg(\Phi(w_n,w))=\deg(w_n)$ and $\phi_n\notin \Phi(w_n, w).$
		\item It is the stable limit of an infinite MLV chain 
		\begin{align*}
			w_0\xrightarrow{\phi_1,\gamma_1} w_1\xrightarrow{\phi_2.\gamma_2}\cdots \longrightarrow w_{i}\xrightarrow{\phi_{i+1},\gamma_{i+1}} w_{i+1}\longrightarrow\cdots .
		\end{align*}
	\end{enumerate}
	We say that the MLV chain of  $w$ has finite length $r,$ quasi-finite length $r,$ or infinite length, respectively.
\end{theorem}
\begin{remark}(\cite[Lemma 4.5]{EN1})\label{mlvtl}
	In the above theorem, the MLV chain of $w$ is of type (i)  if and only if $w$ is either valuation-transcendental or has a nontrivial support.
\end{remark}

The length of an MLV chain of $w$ is called the {\bf depth} of $w,$ denoted by $\text{depth}(w).$ For any MLV chain, the {\bf relative gap} of an augmentation step $w_{i}\longrightarrow w_{i+1}$ is defined as the rational number
\[d_{i}:=\frac{\deg(w_{i+1})}{\deg(\Phi(w_{i},w_{i+1}))}= \begin{cases} 1 & \text{if} \ w_{i}\longrightarrow w_{i+1}\ \text{is ordinary }\\ \deg(w_{i+1})/\deg(w_{i})\ &  \text{if} \ w_{i}\longrightarrow w_{i+1} \ \text{is limit }\end{cases} .\] Observe that $d_i\geq 1.$ In fact, $d_i=1$ if and only if $w_{i}\longrightarrow w_{i+1}$ is an ordinary augmentation. Moreover, the degrees $\deg (w_i),$ depth, relative gaps and the nature of the augmentation steps $w_{i}\longrightarrow w_{i+1}$ are independent of the choice of an MLV chain of $w$ (see \cite[Corollary 4.4, Theorem 4.7]{EN1}).

\subsection{Extensions of valuations}

Let ($\overline{K},\bar{v})$ be as before. If $K^{sep}$ is the separable closure of $K$ in $\overline{K}$, then the fixed field of the {\bf decomposition group}\[D_{\bar{v}}=\{\sigma\in \text{Gal}(K^{sep}|K)\mid\bar{v}\circ\sigma=\bar{v}\},\] denoted by $K^h,$ is called the {\bf Henselization} of $K.$ Let $v^h$ be the restriction of $\bar{v}$ to $K^h.$ Since $\overline{K}|K^{sep}$ is purely inseparable, for all $\sigma\in \text{Aut}(\overline{K}|K),$ we have \[\bar{v}\circ\sigma=\bar{v}\iff \sigma\in \text{Aut}(\overline{K}|K^h).\] Hence, the valuation $v^h$ has a unique extension to $\overline{K}.$

A valued field $(K,v)$ is called {\bf Henselian} if $v$ has a unique extension to every finite extension $L$ of $K,$ i.e., $K=K^h.$ Any finite extension $L$ of $K$ is called {\bf unibranched} if $v$ has a unique extension to $L.$

	\begin{remark}(\cite[Section 17]{O} and \cite[Section 3]{EN3})\label{r21}
	For any irreducible polynomial $g$ over $K,$ the extensions of $v$ to the simple field extension $K[x]/(g)$ are in one-to-one correspondence with the irreducible factors of $g$ in $K^h[x].$ More precisely, if $g=G_1G_2\cdots G_s$ is the factorization of $g$ into irreducibles over $K^h,$ then there are exactly $s$ extensions of $v$ to $L=K[x]/(g)$, given by ${v}_i=\bar{v}\circ \lambda_{\theta_i},$ where $ \lambda_{\theta_i}(x+gK[x])=\theta_i$, and $\theta_i\in Z(G_i)$\footnote{For any $f\in \overline{K}[x],$ $Z(f)$ is the set of all roots of $f$ in $\overline{K}.$}, for all $i,\ 1\leq i \leq s.$
	Moreover, there are  exactly $s$ valuations on $K[x]$ with nontrivial support $gK[x]$ induced by ${v}_i$, defined as 
	\[v_{G_i}:K[x]\twoheadrightarrow K[x]/(g)\xrightarrow{{v}_i} \Gamma_{\bar{v}}\cup\{ \infty \}.\]
\end{remark}

If $w$ is any extension of $v$ to $K[x],$ then there exist a unique common extension $w^h$ of $w$ and $v^h$ to $K^h[x]$ (see \cite[Theorem A]{EN5}). In particular, we have 

\begin{proposition}(\cite[Proposition 5.6]{EN3})\label{df1}
Let $w$ be a valuation transcendental extension. Then for any $f\in\operatorname{KP}(w),$ there exist a unique irreducible factor $f^h$ of $f$ in $K^h[x]$ such that $f^h\in \operatorname{KP}(w^h)$ and $\text{in}_{w^h}(f/f^h)$ is a unit in $\mathcal{G}_{w^h}.$ 
Moreover, we have $\deg_w(f)=\deg_{w^h}(f^h).$ In particular, if $f$ has minimal degree in $\operatorname{KP}(w),$ then $f^h$ has minimal degree in $\operatorname{KP}(w^h).$
\end{proposition}

	Let $\overline{w}$ be any common extension of $w$ and $\bar{v}$ to $\overline{K}[x].$ For any polynomial $f$ in $K[x],$ we define the {\bf optimizing value} of $f$ as
$$\delta(f):=\max\{\overline{w}(x-\alpha)\mid \alpha\in Z(f)\}.$$ A root $\alpha$ of $f$ is said to be an {\bf optimizing root} of $f$ if $\overline{w}(x-\alpha)=\delta(f).$
This optimizing value of $f$ is independent of the choice of the common extension $\overline{w}.$ In particular, we have

  \begin{theorem}(\cite[Corollary 3.13]{EJP})\label{t21}
 	Let $w$ be a valuation-transcendental extension of $v$ to $K[x],$ and $f\in \operatorname{KP}(w).$ Then the set of optimizing roots of $f$ is equal to $Z(f^h).$
 \end{theorem}
 
 \begin{theorem}(\cite[Remark 3.14]{EJP})\label{t2}
 	Let $w$ be valuation-transcendental, and $f\in\operatorname{KP}(w).$ Then for any common extension $\overline{w}$ of $w$ and $\bar{v}$ to $\overline{K}[x]$, we have $\overline{w}=\overline{w}_{\alpha,\delta},$ where $\alpha\in Z(f^h)$ and $\delta=\delta(f).$ Moreover, $\delta=\delta(f)$ is independent of the choice of the key polynomial.
 \end{theorem}
 
 \section{Common extensions and $j$-invariant}
 Let $(K,v), (\overline{K},\bar{v})$ be as before and $w$ a valuation-transcendental extension of $v$ to $K[x].$ In this section, we give an explicit formula for the number of common extensions $\overline{w}$ of $w$ and $\bar{v},$ in terms of $j$-invariant. As a consequence, we obtain some relation between the irreducibility of different key polynomials over $K^h.$
 
 In 2022, Dutta \cite{AD1}  introduced the notion of $j$-invariant. He later generalized it, along with some properties, in \cite{AD2} and \cite{AR}.
 \begin{definition}
  Let $\overline{w}$ be any common extension of $w$ and $\bar{v}$ defined by a pair $(\alpha,\delta)$. For any polynomial $f\in\overline{K}[x],$ the {\bf$j$-invariant} of $f$, denoted by $j(f),$ is the number of roots $\beta$ of $f,$ counting with multiplicities, such that $\bar{v}(\beta-\alpha)\geq \delta.$	
 \end{definition}
    The $j$-invariant is independent of the choice of $\overline{w}$ (see \cite[Theorem 3.1]{AD2}). 
 
\begin{theorem}\label{p23}
	Let $w$ be a valuation-transcendental extension of $v$ to $K[x],$ and $f\in \operatorname{KP}(w).$ If there are $r$ common extensions of $w$ and $\bar{v}$ to $\overline{K}[x],$ then \[r=\frac{\deg f^h}{j(f^h)}.\]
\end{theorem}
\begin{proof}
 By Theorem \ref{t2} we have that, the $r$ common extensions (say)  $\overline{w}_1, \overline{w}_2, \ldots ,\overline{w}_r ,$ are given by\[\overline{w}_1=\overline{w}_{\alpha_1,\delta}, \overline{w}_2=\overline{w}_{\alpha_2,\delta}, \ldots ,\overline{w}_r=\overline{w}_{\alpha_r,\delta},\] where $\alpha_1, \alpha_2, \ldots ,\alpha_r$ are roots of $f^h,$ and $\delta=\delta(f).$\\
Now take an arbitrary $\beta\in Z(f^h),$ then by Theorem \ref{t21}, $\beta$ is an optimizing root of $f.$ Therefore, we have \[\overline{w}_{\alpha_i,\delta}(x-\beta)=\delta,\  \text{for some} \ i,\ 1\leq i\leq r.\]
Since $\overline{w}_{\alpha_i,\delta}(x-\alpha_i)=\delta,$ we have $\bar{v}(\beta-\alpha_i)\geq \delta,$ and hence, $\overline{w}_{\alpha_i,\delta}=\overline{w}_{\beta,\delta}.$ Also for any $j\neq i,$ as $\overline{w}_{\alpha_j,\delta}\neq \overline{w}_{\alpha_i,\delta},$  we have $\bar{v}(\beta-\alpha_j)<\delta.$ Hence, in view of the definition of $j$-invariant, it now follows that \[\deg f^h=r j(f^h).\] 	
\end{proof}
 Observe that by Theorem \ref{t2}, $j(f^h)>0$ for any $f\in\operatorname{KP}(w),$ therefore $(\deg f^h)/j(f^h)$ is well defined. Since $r$ is a constant for $w$, we have 
\begin{corollary}
	Let $w$ be a valuation-transcendental extension of $v$ to $K[x],$ and $f, g\in \operatorname{KP}(w).$ Then \[\frac{\deg f^h}{j(f^h)}=\frac{\deg g^h}{j(g^h)} .\]
\end{corollary}
Note that if $f\in\operatorname{KP}(w)$ and $\beta\in Z(f/f^h),$ then by Theorem \ref{t21}, $\beta$ is not an optimizing root of $f,$ i.e., $\overline{w}_{\beta,\delta}$ is not a common extension of $w$ and $\bar{v},$ which implies that $\bar{v}(\beta-\alpha_i)<\delta,$ for every $i, \ 1\leq i \leq r.$ Hence, we have 
\begin{proposition}\label{dt2}
	Let $w$ be a valuation-transcendental extension of $v$ to $K[x],$ and $f\in \operatorname{KP}(w).$ Then 
	\[j(f)=j(f^h).\]
\end{proposition}
It may be pointed out that the above result is already proved in \cite[Proposition 2.7]{AR}. Moreover, we also know that
\begin{proposition}(\cite[Proposition 3.5]{AD2})\label{dp1}
	Let $w$ be a valuation-transcendental extension of $v$ to $K[x],$ and $\phi\in \operatorname{KP}(w)$ of minimal degree.  Then for any $f\in K[x],$ we have
	\[j(f)=j(\phi)\deg_w(f).\]
\end{proposition}
Using Theorem \ref{p23} and the above two propositions together with Lemma \ref{kpl1}, we obtain 
\begin{proposition}
Let $w$ be a valuation-transcendental extension of $v$ to $K[x],$ and $\phi\in \operatorname{KP}(w)$	of minimal degree. Then for any $f, g\in \operatorname{KP}(w),$ we have \begin{enumerate}[(i)]
	\item $\deg f=\deg g\iff j(f)=j(g)\iff \deg f^h=\deg g^h,$
		\item $\deg f<\deg g\iff j(f)<j(g)\iff \deg f^h<\deg g^h,$
			\item $\deg f\divides\deg g\iff j(f)\divides j(g)\iff \deg f^h\divides\deg g^h.$
\end{enumerate}
\end{proposition}
\begin{proposition}\label{bp3}
	Let $w$ be a valuation-transcendental extension of $v$ to $K[x].$ Then for any $f, g\in \operatorname{KP}(w),$ we have \begin{enumerate}[(i)]
		\item $\deg f^h\divides \deg f \iff \deg g^h\divides \deg g,$
		\item $f=f^h \iff g=g^h.$
	\end{enumerate}
\end{proposition}
\begin{proof}
	Let $\phi\in\operatorname{KP}(w)$ of minimal degree. Then by Proposition \ref{dp1}, we have \[j(f)=j(\phi)\deg_w(f).\] Since $\phi^h\in \operatorname{KP}(w^h) $ of minimal degree (by Proposition \ref{df1}), again on applying the Proposition \ref{dp1} for $w^h$, we have \[j(f^h)=j(\phi^h)\deg_{w^h}(f^h).\] On using Proposition \ref{dt2}, the above two equations implies that\[\deg_w(f)=\deg_{w^h}(f^h).\] Further by Proposition \ref{df1},  $\deg (w)=\deg\phi$ and $\deg (w^h)=\deg\phi^h,$ which in view of Lemma \ref{kpl1}, implies that
	\[\frac{\deg f}{\deg \phi}=\frac{\deg f^h}{\deg \phi^h}.\] Hence, we have 
	\[\frac{\deg f}{\deg f^h}=\frac{\deg \phi}{\deg \phi^h}.\]Arguing similarly, for any $g \in \operatorname{KP}(w),$ we obtain	\[\frac{\deg g }{\deg g^h}=\frac{\deg \phi}{\deg \phi^h}=\frac{\deg f}{\deg f^h}.\] This proves the proposition immediately.
\end{proof}
From Remark \ref{r21}, we have that for any irreducible polynomial $f\in K[x]$, $f$ is irreducible over $K^h$ if and only if $K(\alpha)|K$ is unibranched for every $\alpha\in Z(f).$ In particular, if $f\in \operatorname{KP}(w),$ then $f$ is irreducible over $K^h$ if and only if $f=f^h.$ Therefore, we have 
\begin{proposition}\label{bp1}
	Let $w$ be a valuation-transcendental extension of $v$ to $K[x],$ and $f \in \operatorname{KP}(w).$ Then\[f=f^h \iff K(\alpha)|K \ \text{is unibranched, for every } \ \alpha\in Z(f).\]
\end{proposition}
The above two propositions immediately yield the following:
\begin{proposition}\label{bp2}
	Let $w$ be a valuation-transcendental extension of $v$ to $K[x],$ and $f , g\in \operatorname{KP}(w).$ Then
	\[K(\alpha)|K \ \text{is unibranched} \iff K(\beta)|K \ \text{is unibranched,} \] for every $\alpha\in Z(f)$ and  $\beta \in Z(g).$
	\end{proposition}

	\section{Defect}

	Let $(K,v)$ be a valued field, $L|K$ a finite extension and $v_L$ an extension of $v$ to $L.$ The {\bf Henselian-defect} of the extension $v_L|v$ is defined as \[d(v_L|v)=\frac{[L^h:K^h]}{ef},\] where $e$ and $f$ denote respectively the ramification index and inertia degree of $v_L|v,$  and $L^h, K^h$ are the Henselizations of $L$ and $K$  with respect to the extension of $v_L$ to $\overline{K}.$  If $L|K$ is unibranched, then $v_L$ is the unique extension of $v$ to $L.$ In this case, we denote $d(v_L|v)$ as $d(L|K).$ A finite extension $L|K$ is said to be {\bf defectless} if $d(v_L|v)=1$ for every extension $v_L$ of $v$ to $L.$ A valued field $(K,v)$ is said to be defectless if every finite extension of $K$ is defectless.
	
	In 2007, Vaqui\'e gave a characterization of the defect of a unibranched simple algebraic extension.
	\begin{theorem}(\cite[Corollary 2.10]{V2})\label{vth1}
		Let $v_L$ be the unique extension of $v$ to a simple algebraic extension $L,$ and $w$ be the valuation on $K[x]$ induced by $v_L.$ The defect of the extension $v_L|v$ is the product of the relative gaps of any MLV chain of $w.$
	\end{theorem}
	Later in 2023, Nart and Novacoski generalized this characterization to an arbitrary simple algebraic extension (see Theorem \ref{T2}), using the notion of the defect of an augmentation. To state this generalization, we first recall some definitions and results.
	
	\begin{lemma}(\cite{V0} and \cite[Lemma 6.1]{EN3})
	Let $w\longrightarrow w'$ be an augmentation and $\Phi(w,w')$ be the corresponding tangent direction, equipped with the pre-ordering determined by the action of $w'.$ For $Q\in\Phi(w,w') ,$ set $\rho_Q=[w;Q, w'(Q)].$ Let $\phi \in K[x]$ be either a key polynomial of minimal degree for $w',$ or $supp(w')=\phi K[x].$ Then the positive integer \[d=\min\{\deg_{\rho_Q}(\phi) \mid Q\in \Phi(w,w')\}\] is independent of the choice of $\phi.$ Moreover, the set of all $Q\in \Phi(w,w')$ such that $\deg_{\rho_Q}(\phi)=d$ is cofinal in $\Phi(w,w').$
	\end{lemma}
	This stable value $d$ is called the {\bf defect} of the augmentation and is denoted by $d(w\longrightarrow w').$
	
	\begin{lemma}(\cite[Lemma 6.3]{EN3})\label{nl1}
	If $w\longrightarrow w'$ is an ordinary augmentation, then $d(w\longrightarrow w')=1.$
	\end{lemma}
	\begin{lemma}(\cite[Lemma 6.4]{EN3})\label{nl2}
	Suppose $(K,v)$ is Henselian. For all augmentations $w\longrightarrow w'$ we have
	\[d(w\longrightarrow w')=\deg(w')/ \deg ({\Phi(w,w')}).\]
	\end{lemma}

	Let $w$ be either a valuation-transcendental extension or a valuation with nontrivial support. By Remark \ref{mlvtl}, $w$ has a finite MLV chain (say)
	\begin{align}\label{ed1}
			(v \longrightarrow) w_0  \longrightarrow w_1  \longrightarrow \cdots  \longrightarrow w_{n-1}  \longrightarrow w_n=w .
			\end{align}
			We now recall the definitions of defect, inertia degree, and ramification index for $w|v,$ given in \cite{EN3}.
			The defect of $w|v$ is defined as
			\[d(w|v):=d(w_0\longrightarrow w_1) \cdots d(w_{n-1}\longrightarrow w).\]
		
	 By considering the canonical homomorphisms of graded algebras $\mathcal{G}_{w_i} \longrightarrow \mathcal{G}_{w_{i+1}},$ for every $i,\ 0\leq i\leq n-1,$ the above chain induces a tower of finite field extensions 
	\[k_v=\kappa_{w_0}\longrightarrow\kappa_{w_1}\longrightarrow\cdots\longrightarrow \kappa_{w_{n-1}} \longrightarrow \kappa_{w_n}=\kappa_w,\]
		where $\kappa_{w_i}$ is the relative algebraic closure of $k_v$ in the grade-zero component $\Delta_{w_i}$ of the graded algebra $\mathcal{G}_{w_i}.$ The field $\kappa_{w_i}$ also satisfies $\kappa_{w_i}^\times=\Delta_{w_i}^\times,$ where $\Delta_{w_i}^\times$ is the multiplicative group of all units in $\Delta_{w_i}.$
	
	Hence, we have
	\begin{align*}
		[\kappa_w:k_v]=[\kappa_{w_1}:k_v]\cdots [\kappa_w:\kappa_{w_{n-1}}].
	\end{align*}
	Also, the value groups of the valuations in (\ref{ed1}), form a chain of subgroups
	\[\Gamma_{w_{-1}}:=\Gamma_v \subset \Gamma_{w_0} \subset \cdots \subset \Gamma_{w_{n-1}}\subset \Gamma_{w_n}=\Gamma_w, \]
	such that $\Gamma_{w_{i-1}}=\Gamma^\circ_{w_i}$ for all $i, \ 0\leq i \leq n.$ Hence,
	\[(\Gamma_w:\Gamma_v)=e(w_0)\cdots e(w_n),\] where $e(w_i)=(\Gamma_{w_i}:\Gamma^\circ_{w_i}).$ We call $[\kappa_w:k_v], (\Gamma_w:\Gamma_v)$ the inertia degree and  ramification index of $w|v,$ respectively, which are connected to defect of $w|v$ as follows.

	\begin{theorem}(\cite[Theorem 6.13]{EN3})\label{T1}
		Suppose that $w$ is either a valuation-transcendental extension or has nontrivial support. Then, 
		\[(\Gamma_w:\Gamma_v)[\kappa_w:k_v]d(w|v)=e(w)\deg(w^h).\] 
	\end{theorem}
		The following result relates the defect of an extension of $v$ to the simple finite extension of $K$ with the defect of its induced valuation on $K[x].$
	\begin{theorem}(\cite[Theorem 6.14]{EN3})\label{T2}
		Let $v_L$ be an extension of $v$ to a finite simple extension $L|K.$ Let $w$ be the valuation on $K[x]$ induced by $v_L.$ For any MLV chain of $w=w_n$ (\ref{ed1}), we have\[d(v_L|v)=d(w|v)=d(w_0\longrightarrow w_1)\cdots d(w_{n-1}\longrightarrow w_n).\]
	\end{theorem}
	Note that using Lemmas \ref{nl1} and \ref{nl2}, it follows that Theorems \ref{vth1} and \ref{T2} are equivalent when $(K,v)$ is Henselian.
	
	A valued field $(K,v)$ is called {\bf simply defectless} if all simple algebraic extensions of $K$ are defectless.\\
	We now give a characterization of a simply defectless field in terms of the MLV chains.
	\begin{theorem}
		A valued field $(K,v)$ is simply defectless if the MLV chain of every extension of $v$ to $K[x]$ with nontrivial support consists of only ordinary augmentations.
	\end{theorem}
	\begin{proof}
	Let $\theta \in \overline{K}$ be arbitrary and $v_1, v_2,\ldots, v_s$ be the distinct extensions of $v$ to $K(\theta)\cong K[x]/(g),$ where $g$ is the minimal polynomial of $\theta$ over $K.$ For an arbitrary but fixed $i, \ 1\leq i\leq s,$ let $\bar{v}_i$ be the extension of $v_i$ to the algebraic closure $\overline{K}$ of $K.$ 
	The valuation $v_i$ on $K(\theta)$ determines a valuation $\mu_i$ on $K[x]$\[\mu_i:K[x]\twoheadrightarrow K[x]/(g)\cong K(\theta)\xrightarrow{v_i} \Gamma_{\bar{v}_i}\cup\{ \infty \}, 
	\] having support $gK[x].$ Let the MLV chain of $\mu_i$ be 
	\[w_0\longrightarrow w_1 \longrightarrow \cdots \longrightarrow w_{n-1} \longrightarrow w_n=\mu_i.\] By hypothesis, we have that each augmentation step in the above chain is ordinary. Therefore using Theorem \ref{T2} and Lemma \ref{nl1}, we obtain that	\[
	d(v_i|v) =d(\mu_i|v)=d(w_0\longrightarrow w_1)\cdots d(w_{n-1}\longrightarrow w_n) =1 .\]
	As $i$ was arbitrary, we have $d(v_i|v)=1$ for every $i,1\leq i\leq s.$ Hence, $K(\theta)|K$ is defectless. Since $\theta$ was arbitrary, we have that $(K,v)$ is simply defectless.		
	\end{proof}
	
	The converse of the above theorem also holds for Henselian valued fields $(K,v).$ In fact, we have the following result for every extension of $v$ to $K[x].$
	\begin{theorem}\label{depth}
		For a simply defectless Henselian valued field $(K,v),$ the MLV chain of every extension of $v$ to $K[x]$ consists of only ordinary augmentations.
	\end{theorem}
		\begin{proof}
	 Let $w$ be a valuation on $K[x]$ extending $v.$ By Theorem \ref{mlvt1}, the MLV chain of $w$ is of the form
		\begin{align}\label{fteq1}
			w_0  \xrightarrow{\phi_1,\gamma_1} w_1  \xrightarrow{\phi_2,\gamma_2} \cdots  \longrightarrow w_{i-1}  \xrightarrow{\phi_i,\gamma_i} w_i  \longrightarrow\cdots \longrightarrow w,	
		\end{align}
		where for each $i,$ $w_i(<w)$ is residue-transcendental. It is enough to show that for an arbitrary but fixed $i,$ $w_{i-1}  \xrightarrow{\phi_i,\gamma_i} w_i$ is ordinary. As $\phi_i \in \operatorname{KP}(w_i)$ is of minimal degree, so by Theorem \ref{t2}, $w_i=w_{\alpha_i,\delta_i},$ where $\delta_i$ is the optimizing value of $\phi_i$ and $\alpha_i \in Z(\phi_i^h)$. We define $v_{\phi_i}$ on $K[x]$ as $v_{\phi_i}(f)=\bar{v}(f(\alpha_i)),$ for every $f\in K[x].$ Then $v_{\phi_i}$ is a valuation on $K[x]$ with nontrivial support $\phi_iK[x].$ Note that for any $g\in K[x]$ with $\deg g<\deg \phi_i,$ $w_i(g)=\bar{v}(g(\alpha_i))$ (\cite[Lemma 3.6]{W}), and hence $v_{\phi_i}(g)= w_i(g).$ As $v_{\phi_i}(\phi_i)=\infty$, for any $f\in K[x]$ with $\phi_i$-expansion $f=\sum_{j\geq 0}f_j \phi_i^j$,  we have \[\min_{j\geq 0}\{v_{\phi_i}(f_j\phi_i^j)\}=\min_{j\geq 0}\{w_i(f_j)+j\infty\}=w_i(f_0)=v_{\phi_i}(f_0)=v_{\phi_i}(f).\] Therefore $v_{\phi_i} =[w_i; \phi_i,\infty],$ i.e., $w_i\longrightarrow v_{\phi_i}$ is an ordinary augmentation defined by $\phi_i$ and $\infty.$ By Lemma 4.2 of \cite{EN1}, the MLV chain of $v_{\phi_i}$ is of the form 
		\[ w_0  \xrightarrow{\phi_1,\gamma_1} w_1  \xrightarrow{\phi_2,\gamma_2} \cdots  \longrightarrow w_{i-1}  \xrightarrow{\phi_i,\infty} v_{\phi_i},\] where $w_{i-1}\longrightarrow v_{\phi_i}$ is an ordinary augmentation (respectively limit) if $w_{i-1}\longrightarrow w_i$ is ordinary (respectively limit) in (\ref{fteq1}). Consider the simple extension $L=K[x]/(\phi_i) \cong K(\alpha_i).$ Since $v_{\phi_i}$ is a valuation on $K[x]$ with nontrivial support $\phi_i K[x],$ it is induced by the unique extension $v_L$ of $v$ to $L.$ If $d_j\geq1,\ 0\leq j\leq i-1,$ are the relative gaps for the MLV chain of $v_{\phi_i},$ then by Theorem \ref{vth1}, we have
		\[d(v_L|v)=d_0d_1\cdots d_{i-1}.\]
		By hypothesis, as $K(\alpha_i)|K$ is defectless, i.e., $d(v_L|v)=1,$ it follows that $d_j=1, \forall \ 0\leq j\leq i-1.$ Therefore, every step in the MLV chain of $v_{\phi_i}$ is ordinary. In particular, $ w_{i-1}\longrightarrow v_{\phi_i}$ is ordinary and hence, $w_{i-1}\longrightarrow w_i$ is also ordinary.
			\end{proof}
	Combining the above two theorems, we have
	\begin{corollary}
	A Henselian valued field $(K,v)$ is simply defectless if and only if the MLV chain of every extension of $v$ to $K[x]$ consists of only ordinary augmentations.
	\end{corollary}

	 In the next result, we show that the defect of any valuation-transcendental extension is equal to a Henselian-defect.
		
		\begin{theorem}\label{odt2}
			Let $w$ be a valuation-transcendental extension of $v$ to $K(x),$ and $\phi\in \operatorname{KP}(w)$ of minimal degree. Then $d(w|v)=d(L^h|K^h),$ where $L=K[x]/(\phi).$ 
			Moreover, $d(w|v)$ is independent of the choice of minimal degree key polynomial $\phi.$
		\end{theorem}
		\begin{proof}
			Since $w$ is a valuation-transcendental extension, the MLV chain of $w$ is of the form 
			\begin{align}\label{te1}
				w_0\longrightarrow w_1 \longrightarrow \cdots \longrightarrow w_{n-1} \longrightarrow w_n=w.
				\end{align} For any optimizing root $\alpha \in Z(\phi^h),$ the map\[v_{\phi} : K[x]\longrightarrow \Gamma_{\bar{v}}\cup \{\infty\}\  \text{defined as} \  v_{\phi}(f)=\bar{v}(f(\alpha)), \ \forall \ f\in K[x],\] is a valuation on $K[x]$ having nontrivial support $\phi K[x].$ Arguing as in the proof of the above theorem, we have that $v_{\phi} =[w; \phi,\infty]$ and the MLV chain of $v_{\phi}$ is
			\begin{align}\label{e2}
				w_0\longrightarrow w_1 \longrightarrow \cdots \longrightarrow w_{n-1} \longrightarrow v_{\phi}. \end{align} 
				Now for the chains (\ref{te1}) and (\ref{e2}), we have \begin{align}\label{e3}
				[\kappa_w:k_v]=[\kappa_{w_1}:k_v]\cdots [\kappa_w:\kappa_{w_{n-1}}],\ 
				[\kappa_{v_\phi}:k_v]=[\kappa_{w_1}:k_v]\cdots [\kappa_{v_\phi}:\kappa_{w_{n-1}}]
				\end{align}
				and
				\begin{align}\label{e4}
					(\Gamma_w:\Gamma)=e(w_0)\cdots e(w_{n-1})e(w),\ 
					 (\Gamma_{v_\phi}:\Gamma)=e(w_0)\cdots e(w_{n-1})e(v_\phi).
				\end{align} 
			 From Proposition 3.6 of \cite{EN4}, we have $\kappa_w \cong k_{v_\phi}.$ Since $v_\phi$ has nontrivial support, by Theorem 2.2 of \cite{EN3}, $\mathcal{G}_{v_\phi}=\mathcal{G}^\circ_{v_\phi}$ which implies that $\kappa_{v_\phi}=\Delta_{v_\phi}.$ As we know that the residue field $k_{v_\phi}$ is the field of fraction of $\Delta_{v_\phi}$ (see \cite[Section 4]{EN4}), we have $\Delta_{v_\phi}=k_{v_\phi}.$ Therefore, $\kappa_w \cong \kappa_{v_\phi},$ which implies that $[\kappa_w:\kappa_{w_{n-1}}]=[\kappa_{v_\phi}:\kappa_{w_{n-1}}],$ and hence by (\ref{e3}), $[\kappa_w:k_v]=[\kappa_{v_\phi}:k_v].$
			As $\phi \in \operatorname{KP}(w)$ of minimal degree, and $\phi K[x]=\text{supp}(v_\phi),$ we have $\deg(w)=\deg\phi=\deg (v_\phi).$ By Theorem A of \cite{EN5}, there exist unique extensions $w^h$ and $v_{\phi}^h$ of $w$ and $v_\phi$ respectively to $K^h[x].$ So by Proposition \ref{df1} and Lemma 5.7 of \cite{EN3}, we have that $\phi^h \in \operatorname{KP}(w^h)$ of minimal degree, and $\phi^h K^h[x]=\text{supp}(v_{\phi}^h),$ hence $\deg (w^h)=\deg\phi^h=\deg (v_{\phi}^h).$ On applying Theorem \ref{T1} to the valuations $w$ and $v_\phi$, we have
			\[(\Gamma_w:\Gamma)[\kappa_w:k_v]d(w|v)=e(w)\deg(w^h),\]and
			\[(\Gamma_{v_\phi}:\Gamma)[\kappa_{v_\phi}:k_v]d(v_\phi|v)=e(v_\phi)\deg(v_{\phi}^h).\] On dividing the above two equations and using (\ref{e4}), we obtain $d(w|v)=d(v_\phi|v).$ By Remark \ref{r21} and Theorem \ref{T2}, there exist an extension $v_i$ of $v$ to $L$ (the restriction of $\bar{v}$ to L) such that $d(v_\phi|v)=d(v_i|v)=d(L^h|K^h).$ Hence, the result follows.
	    \end{proof}
	    \begin{corollary}\label{2.11}
	    	Let $w$ be a valuation-transcendental extension of $v$ to $K(x),$ $\phi\in K[x]$ a minimal degree key polynomial for $w$ and $\alpha \in Z(\phi).$ If $K(\alpha)|K$ is unibranched, then  $d(w|v)=d(K(\alpha)|K).$ 
	    \end{corollary}
	  Using the above corollary, we now prove that the defect of a unibranched simple algebraic extension generated by a root of any key polynomial for $w$ is an invariant for $w.$
	    \begin{theorem}\label{2.12}
	    	Let $w$ be a valuation-transcendental extension of $v$ to $K(x),$ and $f,g\in \operatorname{KP}(w)$ such that  $f=f^h.$ Then  \[d(K(\theta)|K)=d(K(\eta)|K), \]	for every $\theta\in Z(f), \ \text{and} \ \eta \in Z(g).$
	    \end{theorem}
	    \begin{proof}Let $\phi$ be a minimal degree key polynomial for $w,$ and $\alpha\in Z(\phi)$ be arbitrary. As $f=f^h$, using Propositions \ref{bp1} and \ref{bp2} we have, $K(\alpha)|K, K(\theta)|K$ and $K(\eta)|K$ are all unibranched. If $w$ is value-transcendental, then $f,g\in \operatorname{KP}(w)$ are also of minimal degree, therefore by above corollary, we have that \[d(K(\theta)|K)=d(K(\eta)|K), \] for every $\theta\in Z(f)$ and $\eta \in Z(g).$ Suppose now that $w$ is residue-transcendental. If $f$ and $g$ are minimal degree key polynomials for $w,$ then the result follows immediately from the above corollary. Therefore, assume that $\deg f>\deg \phi,$ and define an ordinary augmentation $w'=[w;f,\gamma]$ for some $\gamma\in\Gamma$ such that $\gamma>w(f).$ Then $f\in \operatorname{KP}(w')$ of minimal degree, which implies that \[\deg (w')=\deg f=\deg (\Phi(w,w'))>\deg \phi=\deg (w),\] i.e., $w\longrightarrow w'$ is an MLV step. Therefore, from the MLV chain of $w,$  we obtain the MLV chain of $w'$ (say)
	    	\[w_0\longrightarrow w_1 \longrightarrow \cdots \longrightarrow w_{n-1} \longrightarrow w_n=w\longrightarrow w'.\]
	    	Clearly, $d(w'|v)=d(w|v)d(w\longrightarrow w').$ As $w\longrightarrow w'$ is an ordinary augmentation, by Lemma \ref{nl1}, $d(w\longrightarrow w')=1.$	It follows that
	    	\begin{align}\label{e5}
	    		d(w'|v)=d(w|v).\end{align}
	    	Now, using the above corollary, we have that 
	    	\[d(w'|v)=d(K(\theta)|K), \ \forall \ \theta \in Z(f),\]  and \[d(w|v)=d(K(\alpha)|K), \ \forall \ \alpha \in Z(\phi).\]	Hence, from (\ref{e5}) we have\[d(K(\theta)|K)=d(K(\alpha)|K), \] for every $\theta\in Z(f), \ \text{and} \ \alpha \in Z(\phi).$ Arguing similarly for $g\in\operatorname{KP}(w),$ we have
	    	\[d(K(\eta)|K)=d(K(\alpha)|K), \] for every $\eta\in Z(g), \ \text{and} \ \alpha \in Z(\phi).$ Hence, the result follows from the above two equations.
	    \end{proof}
	    
	    For any $\theta\in\overline{K},$ the map  $v_\theta:K[x]\longrightarrow \Gamma$ given by \[v_\theta(g)=\bar{v}(g(\theta)) \ \text{for every} \ g\in K[x],\] is a valuation on $K[x]$ with nontrivial support $fK[x],$ where $f$ is the minimal polynomial of $\theta$ over $K.$ Then by the {\bf depth} of $\theta,$ we mean the length of the MLV chain of the valuation $v_\theta$ i.e.,\[\text{depth}(\theta):=\text{depth}(v_\theta).\]
	    
	    Let $w$ be a valuation-transcendental extension and $\phi$ be any minimal degree key polynomial for $w$. Then from the proof of Theorem \ref{depth}, we have depth($w$)$=$depth($v_\alpha$) for every optimizing root $\alpha$ of $\phi.$ Now if $f\in\operatorname{KP}(w)$ not of minimal degree, then for the ordinary augmentation $w'=[w;f,\gamma]$ for some $\gamma\in\Gamma$ such that $\gamma>w(f),$ we have that $f$ is the minimal degree key polynomial for $w'$ and $\text{depth}(w')=\text{depth}(w) +1.$ Keeping this in mind, we have
	    \begin{proposition}
	    	Let $w$ be a  valuation-transcendental extension and $f$ be any key polynomial for $w.$ Then for any optimizing root $\theta$ of $f,$ we have
	    	\[\textup{depth}(\theta)=\begin{cases}
	    		\textup{depth}(w) &\ \text{if}\  \deg f=\deg(w)\\
	    		\textup{depth}(w)+1 &\ \text{if}\  \deg f>\deg(w)
	    	\end{cases}.\] 
	    \end{proposition}
	     In 2026, Nart gave the description of key polynomials in terms of minimal polynomials of elements of some suitable balls (see \cite[Proposition 5.1 and Theorem 5.5]{EJP}). Therefore, the above proposition can also be interpreted as 
	     \begin{proposition}
	     	For any pair $(\alpha,\delta)\in \overline{K}\times\Gamma,$ consider the sets $A=\min_KB(\alpha,\delta)\footnote{For any set $S\subset \overline K,$ $\min_K(S)$ is the subset of $S$ containing all minimal degree elements of $S$ over $K.$  }$ and $B=\{\theta\in B(\alpha,\delta)\mid \theta\in \min_KB^\circ(\theta,\delta)\}.$ Then we have \[\ \textup{depth}(\beta)=m,\ \forall \ \beta\in A \ \text{and} \ \textup{depth}(\theta)=m+1, \ \forall \ \theta\in B\backslash A,\] for some nonnegative integer $m.$ Moreover, $m=\textup{depth}(w),$ where $w=w_{\alpha,\delta}.$
	     \end{proposition} 
	    
	\section{Distinguished Pairs} 
	Throughout the section, $w$ is a residue-transcendental extension. Here, we give a characterization of distinguished pairs in terms of key polynomials and generalize some results of distinguished pairs which were earlier known for Henselian valued fields. 
	
	We first recall the fundamental theorem of residue-transcendental extension (see \cite[Theorem2.1]{APZ1} and \cite[Theorem 1.4]{SK}).
	\begin{theorem}\label{dpth1}
Let $(K,v), (\overline{K},\bar{v})$ be as before and $w$ a residue-transcendental extension of $v$ to $K(x).$  Let $\overline{w}$ be a common extension of $w$ and $\bar{v}$ to $\overline{K}(x)$ with  minimal pair of definition $(\alpha,\delta)\in\overline{K}\times \Gamma_{\bar{v}}.$ Let $\phi$ be the minimal polynomial of $\alpha$ over $K$ of degree  $n$ with $w(\phi)=\gamma,$ and $\Gamma_\alpha, k_\alpha$ denote respectively the value group and residue field of $\bar{v}$ restricted to $K(\alpha).$ Then the following hold:
\begin{enumerate}[(i)]
	\item For a non-zero polynomial  $g$  in $K[x]$ of degree less than  $n,$  we have  $w(g)=\bar{v} (g(\alpha)).$
	\item For any  polynomial $f$ in $K[x]$ with  $\phi$-expansion $\sum_{i\geq 0}f_i \phi^i,$ $\deg f_i< n,$ we have 
	$$w(f)=\min_{i\geq 0}\{\bar{v}(f_i(\alpha))+i\gamma \}.$$
	\item Let $e$ be the smallest positive integer such that $e\gamma\in \Gamma_{\alpha}.$ Then there exists a polynomial $h \in K[x]$ of  degree less than $n,$ such that $w(h)=\bar{v}(h(\alpha))=e\gamma,$    the
	$w$-residue, $ r^{*},$ of $\phi^{e}/h$ is transcendental over $k_{\alpha}$ and $k_{w}=k_{\alpha}(r^*).$ 
\end{enumerate}
\end{theorem}
Using the canonical homomorphism from the valuation ring $O_v$ of $v$ onto its residue field $k_v,$ we can lift any monic polynomial  with coefficients in $k_v$ to obtain a monic polynomial with coefficients in $O_v.$ In 1995, Popescu and Zaharescu \cite{PZ} extended this notion using $(K,v)$-minimal pairs for local  fields, which also holds for arbitrary valued fields. 
\begin{definition} 
Let $w$ be a residue-transcendental extension of $v$ to $K(x).$ With notations and hypothesis as in Theorem \ref{dpth1}, a monic  polynomial $f$ in $K[x]$ is said to be a {\bf lifting} of a monic polynomial $U(Y),$ in an indeterminate $Y=r^*$ over $k_\alpha$ having degree $m\geq 1,$ with respect to $(\alpha,\delta)$ (or with respect to $w$), if the following conditions are satisfied:
\begin{enumerate}[(i)]
	\item $\deg f=emn,$
	\item $w(f)=w(h^{m})=em\gamma,$
	\item $w$-residue of $\frac{f}{h^m}$ is equal to $U(Y).$ 
\end{enumerate}
\end{definition}
The lifting $f$ of $U(Y)$ is called trivial if  $\deg f=\deg \phi,$ i.e., $\deg U(Y)=1$ and $\gamma=w(\phi)\in\Gamma_\alpha,$ where $\phi$ is the minimal polynomial of $\alpha$ over $K.$

\begin{definition}
A pair $(\theta,\alpha)$ of elements of $\overline{K}$ is called a $(K,v)$-{\bf distinguished pair}  if the following conditions are satisfied:
\begin{enumerate}[(i)]
	\item $\deg\theta>\deg\alpha,$
	\item $\bar{v}(\theta-\alpha)=\max\{\bar{v}(\theta-\beta)\mid \beta\in\overline{K},\, \deg\beta<\deg\theta \},$
	\item if $\eta\in\overline{K}$ is such that $\deg\eta<\deg\alpha,$ then $\bar{v}(\theta-\eta)<\bar{v}(\theta-\alpha).$
\end{enumerate}
\end{definition}
Clearly for $\delta=\bar{v}(\theta-\alpha)$ and $\delta'>\delta,$ $(\alpha,\delta)$ and $(\theta,\delta')$ are both $(K,v)$-minimal pairs. Moreover, from the definition it also follows that $\overline{w}_{\alpha,\delta} (g)=\bar{v}(g(\theta))$ for any polynomial $g\in K[x]$ of degree less than $\deg \theta.$

For any two irreducible polynomials $f$ and $g$ over $K,$ we call $(g,f)$ a distinguished pair, if there exists a root $\theta$ of $g$ and a root $\alpha$ of $f$ such that $(\theta,\alpha)$ is a $(K,v)$-distinguished pair.

Distinguished pairs give rise to distinguished chains in a natural manner. A chain $\theta=\theta_n,\theta_{n-1},\ldots,\theta_0$ of elements of $\overline{K}$ is called a {\bf saturated distinguished chain}  for $\theta$  of length $n,$ if $(\theta_{i+1},\theta_{i})$ is a $(K,v)$-distinguished pair for $0\leq i\leq n-1$ and $\theta_0\in K.$  Similarly, a chain $\phi=\phi_n,\phi_{n-1},\ldots,\phi_0$ of polynomials in $K[x]$ is called a saturated distinguished chain for $\phi$  of length $n,$ if $(\phi_{i+1},\phi_{i})$ is a distinguished pair for $0\leq i\leq n-1$ and $\phi_0 $ has degree one over $K.$\\
The  following results give characterization of key polynomials and distinguished pairs using liftings of irreducible polynomials.
\begin{theorem}(\cite[Theorem 4.6]{PP})\label{ppth}
Let $w$ be a residue-transcendental extension of $v$ to $K(x).$ With notations and hypothesis as in Theorem \ref{dpth1}, a polynomial $f\in K[x]$, with $\deg f> \deg \phi,$ is a key polynomial for $w$ if and only if it is a nontrivial lifting of some monic irreducible polynomial over $k_\alpha,$ whose constant term is nonzero, with respect to $(\alpha,\delta)$ and $h.$
\end{theorem}

\begin{theorem}(\cite[Theorem 1.2]{AN1})
Let $\theta,$ $\alpha$ be elements in the algebraic  closure $\overline{K}$ of a Henselian valued field $(K,v)$ with respective minimal polynomials $f,$ $\phi$ over $K.$ Suppose that $\deg f>\deg \phi.$ Let $e$ be the smallest positive integer such  that $e\bar{v}(\phi(\theta))\in\Gamma_\alpha$ with $e\bar{v}(\phi(\theta))=\bar{v}(h(\alpha)),$ for some $h\in K[x]$ with $\deg h<\deg \phi.$ Then  $(\theta,\alpha)$ is a $(K,v)$-distinguished pair if and only if $(\alpha,\bar{v}(\theta-\alpha))$ is a $(K,v)$-minimal pair and $f$ is a lifting of some monic irreducible polynomial $U(Y)\neq Y\in k_\alpha[Y]$ with respect to $(\alpha,\bar{v}(\theta-\alpha))$ and the polynomial $h.$
\end{theorem}


 We now generalize the above result for an arbitrary valued field.
 \begin{theorem}\label{3.9}
 	Let $\theta,$ $\alpha$ be elements in the algebraic  closure $\overline{K}$ of a valued field $(K,v)$ with respective minimal polynomials $f,$ $\phi$ over $K.$ Suppose that $\deg f>\deg \phi.$ Then $(\theta,\alpha)$ is a $(K,v)$-distinguished pair if and only if for $\delta=\bar{v}(\theta-\alpha),$ $(\alpha,\delta)$ is a $(K,v)$-minimal pair and $f\in \operatorname{KP}(w),$ where $w={w}_{\alpha,\delta}.$	
 \end{theorem}
\begin{proof}
Suppose first that $(\theta,\alpha)$ is a $(K,v)$-distinguished pair, then for $\delta=\bar{v}(\theta-\alpha)$ and $\delta'>\delta,$ $(\alpha,\delta)$ and $(\theta,\delta')$ are $(K,v)$-minimal pairs. Let $w=w_{\alpha,\delta}$ and $w'=w_{\theta,\delta'},$ then $\phi$ is a minimal degree key polynomial for $w.$ As $\delta'>\delta,$ we have $B(\theta,\delta')\subset B(\alpha,\delta),$ which implies that $\overline{w}_{\alpha,\delta}<\overline{w}_{\theta,\delta'}.$
	Since $(\theta,\delta')$ is a minimal pair of definition for $\overline{w}_{\theta,\delta'},$ by \cite[Lemma 3.6]{W}, we have $w'(g)=\overline{w}(g)=\bar{v}(g(\theta))$ for any polynomial $g\in K[x]$ of degree less than $\deg f.$ Also as $(\theta,\alpha)$ is a distinguished pair, $w(g)=\bar{v}(g(\theta)),$ therefore we have $w(g)=w'(g).$ Let $f=\prod_{i=1}^{m}(x-\theta_i),$  where $\theta_i$ runs over all roots of $f$ counted with multiplicity. Then
	\[w(f)=\sum_{i=1}^m\overline{w}_{\alpha,\delta}(x-\theta_i),\ \ \text{and}\  \ 
	w'(f)=\sum_{i=1}^m\overline{w}_{\theta,\delta'}(x-\theta_i).\]
	As $\overline{w}_{\alpha,\delta}(x-\theta_i)\leq \overline{w}_{\theta,\delta'}(x-\theta_i)$ for every $\theta_i \in Z(f)$ and $ \overline{w}_{\alpha,\delta}(x-\theta)=\delta<\delta'= \overline{w}_{\theta,\delta'}(x-\theta),$ we have $ w(f)<w'(f),$ hence $f\in \Phi(w,w').$ Since $\Phi(w,w')\subset \operatorname{KP}(w),$ therefore, $f$ is a key polynomial for $w.$
	
	Conversely assume that $(\alpha,\delta)$ is a minimal pair, $w=w_{\alpha,\delta},$  and $f\in \operatorname{KP}(w).$ It follows from Theorem 5.5 of \cite{EJP} that there exist $\theta'\in Z(f)$ such that $\theta'\in B(\alpha,\delta) \cap\min_K(B^\circ(\theta',\delta)),$ i.e., $\bar{v}(\theta'-\alpha)\geq \delta$ and 
	\begin{align}\label{dpeq1}
		\bar{v}(\theta'-\gamma)\leq \delta, \ \forall \ \gamma\in\overline{K}, \ \text{with} \ \deg \gamma <\deg\theta'.\end{align}
	Since both $\theta,\theta'\in B(\alpha,\delta)$ we have $(\theta,\delta)$ and $(\theta',\delta)$ are pairs of definition for $w,$ which implies that $\theta,\theta'\in Z(f^h)$ (by Theorem \ref{t21}). Hence, $\theta$ is a $K^h$-conjugate of $\theta',$ i.e., $\theta = \sigma(\theta')$ for some $\sigma\in\text{ Aut}(\overline{K}|K^h).$ Now for any $\gamma\in \overline{K}$ with $\deg \gamma<\deg\theta,$ we have
	\[\bar{v}(\theta-\gamma)=\bar{v}(\sigma(\theta')-\gamma)=\bar{v}\circ\sigma(\theta'-\sigma^{-1}(\gamma))=\bar{v}(\theta'-\sigma^{-1}(\gamma))\leq \delta,\] where the last inequality follows from (\ref{dpeq1}). Therefore $\theta \in \min_KB^\circ(\theta,\delta)$. Using the fact that $(\alpha,\delta)$ is a minimal pair, it can be easily verified that $\alpha$ is of the least degree in $\overline{K}$ such that
	\[\bar{v}(\theta-\alpha)=\max\{\bar{v}(\theta-\gamma)\mid \gamma \in \overline{K}, \ \deg \gamma < \deg \theta\}.\] Hence, $(\theta,\alpha)$ is a $(K,v)$-distinguished pair.
	\end{proof}

Arguing as in the above proof, it can be easily shown that 
\begin{corollary}
	Let $w$ be a residue-transcendental extension, $\phi, f\in \operatorname{KP}(w)$ such that $\phi$ is of minimal degree and $\deg f>\deg \phi.$ Then  $(\theta',\alpha')$ is a $(K,v)$-distinguished pair, for every optimizing roots  $\theta',\alpha'$ of $f$ and $\phi$ respectively with respect to some common extension of $w$ and $\bar{v}.$
	\end{corollary}
	\begin{corollary}
		 	Let $K, v, w, \phi,$ and $f$ be as in the above corollary and assume that $\phi=\phi^h.$ Then for every $\theta \in Z(f)$ there exists an $\alpha \in Z(\phi)$ such that $(\theta,\alpha)$ is a
		$(K,v)$-distinguished pair, and vice-versa.
	\end{corollary}
	
	Note that the forward part of Theorem \ref{3.9} implies that if $(\theta,\alpha)$ is a $(K,v)$-distinguished pair, then there exist a valuation $w$ on $K(x)$  such that their respective minimal polynomials $f,$ $\phi$ over $K$ are key polynomials for $w$ and $\phi$ is of minimal degree. The same is also proved in \cite[Lemma 3.1]{AR}, with a different approach.
	
	To prove the next results of this section, we need the following three lemmas; the first is well known (\cite[Lemma 2.1(ii)]{SK}), the second is analogous to the first, and the third is a slight generalization of Lemma 2.3 of \cite{AK2}.
	
	\begin{lemma}\label{dpl12}
		Let $\overline{w}_{\alpha,\delta}$ be a valuation on $\overline{K}[x]$ defined by $(K,v)$-minimal pair $(\alpha,\delta).$  If $f\in K[x]$ is a polynomial such that for each root $\beta$ of $f,$ $\bar{v}(\alpha-\beta)<\delta,$ then $\overline{w}_{\alpha,\delta}(f(x)-f(\alpha))>\bar{v}(f(\alpha)).$	
	\end{lemma}
	
	\begin{lemma}\label{3.12}
		Let $(\alpha,\delta)$ be a $(K,v)$-minimal pair and $\theta\in \overline{K}$ with $\bar{v}(\theta-\alpha)\geq\delta.$ Let $f\in K[x]$ be a polynomial such that for each root $\gamma$ of $f,$ $\bar{v}(\alpha-\gamma)<\delta.$ Then $\bar{v}(f(\theta)-f(\alpha))>\bar{v}(f(\alpha)).$
	\end{lemma}
	
	\begin{lemma}\label{3.13}
		Let $f$ and $g$ be two monic irreducible polynomials over $K$ of degree $m$ and $n$ respectively, such that $f(\alpha)=g(\beta)=0.$ If $K(\alpha)|K$ and $K(\beta)|K$ are both unibranched, then $n\bar{v}(f(\beta))=m\bar{v}(g(\alpha)).$
	\end{lemma}
	\begin{proof}
		Let $f=\prod_{i=1}^m(x-\alpha_i)$ and $g=\prod_{j=1}^n(x-\beta_j),$ where $\alpha_1=\alpha,$ $\beta_1=\beta$ and $\alpha_i, \beta_j$ runs over all roots of $f$ and $g$ respectively, counted with multiplicity. Since $K(\alpha)|K$ and $K(\beta)|K$ are both unibranched, it follows that $f$ and $g$ are irreducible over $K^h,$ which further implies that all $\alpha_i$'s are $K^h$-conjugate of $\alpha$ and $\beta_j$'s are $K^h$-conjugate of $\beta.$ Hence, for any $j,\ 1\leq j\leq n,$ $\beta_j=\sigma(\beta)$ for some $\sigma\in \text{Aut}(\overline{K}|K^h),$ and hence, \[\bar{v}(f(\beta_j))=\bar{v}(f(\sigma(\beta)))=\bar{v}\circ \sigma(f(\beta))=\bar{v}(f(\beta)).\] Similarly, for all $i,\ 1\leq i\leq m,$ we have $\bar{v}(g(\alpha_i))=\bar{v}(g(\alpha)).$	Keeping in mind that \[
			\prod_{j=1}^{n}f(\beta_j)=\pm	\prod_{i=1}^{m}g(\alpha_i),\]
			we have  \[\bar{v}(	\prod_{j=1}^{n}f(\beta_j))=\bar{v}(\prod_{i=1}^{m}g(\alpha_i)).\]
			Hence,\[ \sum_{j=1}^n\bar{v}(f(\beta_j))=\sum_{i=1}^m\bar{v}(g(\alpha_i)),\]
			which implies that $n\bar{v}(f(\beta))=m\bar{v}(g(\alpha)).$
		\end{proof}
	
	\begin{theorem}\label{dpmt}
	Let $\theta,$ $\alpha$ be elements in the algebraic  closure $\overline{K}$ of a valued field $(K,v)$ with respective minimal polynomials $f,$ $\phi$ over $K.$ Suppose that $\deg f>\deg \phi,$ $K(\theta)|K,$ and $K(\alpha)|K$ are both unibranched. Then $(\theta,\alpha)$ is a $(K,v)$-distinguished pair if and only if the following four conditions are satisfied.
	\begin{enumerate}[(i)]
		\item $(\alpha,\bar{v}(\theta-\alpha))$ is a minimal pair, and $\bar{v}(\theta-\alpha_i)\leq\bar{v}(\theta-\alpha)$ for all $K$-conjugates $\alpha_i$ of $\alpha,$
	
		\item $\Gamma_\theta=\Gamma_\alpha+\mathbb{Z}\bar{v}(\phi(\theta)),$
		\item $k_\theta=k_\alpha\left(\left(\frac{\phi(\theta)^e}{h(\alpha)}\right)^*\right),$ where $e$ is the smallest positive integer such that $e\bar{v}(\phi(\theta))=\bar{v}(h(\alpha))\in \Gamma_\alpha,$ for some $h\in K[x]$ with $\deg h<\deg \phi,$
			\item $d(K(\theta)|K)=d(K(\alpha)|K).$
	\end{enumerate}
	\end{theorem}
	\begin{proof}
	We first assume that $(\theta,\alpha)$ is a distinguished pair. Then assertion (i) follows from the definition of a distinguished pair. Let $\delta=\bar{v}(\theta-\alpha)$ and define a valuation $w=w_{\alpha,\delta}.$ By Theorem \ref{3.9}, we have that $f,\phi\in \operatorname{KP}(w)$ and $\phi$ is of minimal degree.
	Since $\deg f>\deg \phi,$ $w$ is residue-transcendental. Let $e$ be the smallest positive integer such that $ew(\phi)\in \Gamma_{\alpha}.$ Then there exists a polynomial $h \in K[x]$ of  degree less than $\deg \phi,$ such that $w(h)=\bar{v}(h(\alpha))=ew(\phi),$ and therefore by Theorem \ref{ppth}, we have that $f$ is a lifting of a monic irreducible polynomial of degree (say) $t$ over $k_\alpha$ with respect to $(\alpha,\delta)$ and $h.$ Hence $\deg f= et\deg \phi$ i.e.,
	\begin{align}\label{dpeq2}
		 \frac{[K(\theta):K]}{[K(\alpha):K]}=et.
	\end{align} From  Lemma \ref{3.12} , for any polynomial $g\in K[x]$ with $\deg g<\deg \phi$, we have $\bar{v}(g(\theta)-g(\alpha))>\bar{v}(g(\alpha)),$ i.e., $\bar{v}(g(\alpha))= \bar{v}(g(\theta))$ and $(g(\alpha))^*=(g(\theta))^*.$ Therefore, we have \[\Gamma_\alpha \subset \Gamma_\theta, \ \text {and} \ k_\alpha\subset k_\theta.\]Since $\deg f>\deg \phi$, $w(\phi)=\bar{v}(\phi(\theta))$ and $f$ is a proper key polynomial for $w.$ Keeping in mind that $(\theta,\alpha)$ is a distinguished pair, on using Lemma 2.11 and Corollary 6.4 of \cite{EN4}, we have
	\[\Gamma_w=\Gamma^\circ_w+\mathbb{Z}w(\phi),\] 
	\[\text{where}\ \ \Gamma^\circ_w=\{w(g)=\bar{v}(g(\alpha))\mid (0\neq) g\in K[x], \deg g <\deg \phi\}= \Gamma_\alpha,\] 
	\[\text{and}\ \ \Gamma_w=\{w(g)=\bar{v}(g(\theta))\mid (0\neq) g\in K[x], \deg g<\deg f\}= \Gamma_\theta.\]
  Hence, assertion (ii)  follows.
	 As $w$ is residue-transcendental, $\Gamma_w/\Gamma^\circ_w=\Gamma_\theta/\Gamma_\alpha$ is torsion. Since $e$ is the smallest positive integer such that 
	$e\bar{v}(\phi(\theta))=\bar{v}(h(\alpha))\in \Gamma_\alpha,$ \begin{align}\label{dpeq3}
	(\Gamma_\theta:\Gamma_\alpha)=e.	
	\end{align}
 As both $K(\theta)|K$ and $K(\alpha)|K$ are unibranched, we have
	\begin{align}\label{dpeq4}
	[K(\theta):K]=(\Gamma_\theta:\Gamma_v)[k_\theta:k_v]d(K(\theta)|K),	
	\end{align}
	\begin{align}\label{dpeq5}
	[K(\alpha):K]=(\Gamma_\alpha:\Gamma_v)[k_\alpha:k_v]d(K(\alpha)|K),	
	\end{align}
	and using  Theorem \ref{2.12}, $d(K(\theta)|K)=d(K(\alpha)|K),$ i.e., assertion (iv) follows.
	Since $(\Gamma_\theta:\Gamma_v)=(\Gamma_\theta:\Gamma_\alpha)(\Gamma_\alpha:\Gamma_v),$ and $[k_\theta:k_v]=[k_\theta:k_\alpha][k_\alpha:k_v],$ on dividing (\ref{dpeq4}) by (\ref{dpeq5}), we obtain
	\begin{align*}
		\frac{[K(\theta):K]}{[K(\alpha):K]}=(\Gamma_\theta:\Gamma_\alpha)[k_\theta:k_\alpha],
			\end{align*}
			which implies that $et=e[k_\theta:k_\alpha]$ (by  (\ref{dpeq2}) and (\ref{dpeq3})),
		hence $	t=[k_\theta:k_\alpha]$.
	So to prove assertion (iii), it is enough to show that $\left(\frac{\phi(\theta)^e}{h(\alpha)}\right)^*\in k_\theta$ is algebraic over $k_\alpha$ of degree $t.$ Suppose that $\left(\frac{\phi(\theta)^e}{h(\alpha)}\right)^*$ is algebraic over $k_\alpha$ of degree $s<t.$ Then there exist polynomials $q_i\in K[x]$ with $\deg q_i < \deg \phi$, and $q_0(\alpha)^*\neq 0$ such that
	\[\left(\left(\frac{\phi(\theta)^e}{h(\alpha)}\right)^*\right)^s + q_{s-1}(\alpha)^*\left(\left(\frac{\phi(\theta)^e}{h(\alpha)}\right)^*\right)^{s-1}+\cdots+q_0(\alpha)^*=0.\] For $0\leq i \leq {s-1},$ we write $q_i(\alpha)/h(\alpha)^i$ as $g_i(\alpha)$ and $h(\alpha)^{-s}$ as $g_s(\alpha),$ where each $g_i\in K[x]$ is of degree less than $\deg \phi.$ So the above equation can be rewritten as 
	\[(g_s(\alpha)\phi(\theta)^{es})^*+(g_{s-1}(\alpha)\phi(\theta)^{e{(s-1)}})^*+\cdots+g_0(\alpha)^*=0.\] 
	Since $\deg g_i<\deg \phi,$ we have $g_i(\alpha)^*=g_i(\theta)^*$ (Lemma \ref{3.12} ). Therefore, the above equation shows that
	\begin{align}\label{dpeq6}
		\bar{v}(g_s(\theta)\phi(\theta)^{es}+g_{s-1}(\theta)\phi(\theta)^{e(s-1)}+\cdots +g_0(\theta))>0.\end{align}
	Put $g(x)=g_s(x)\phi(x)^{es}+g_{s-1}(x)\phi(x)^{e(s-1)}+\cdots +g_0(x).$ Since $g$ is in its $\phi$-expansion, we have \begin{align}\label{dpeq7}
			w(g)=\min_{0\leq i \leq s}\{\bar{v}(g_i(\alpha))+iew(\phi)\}\leq \bar{v}(g_0(\alpha))=0.
	\end{align}
		  Observe that $\deg g<es\deg \phi+ \deg \phi\leq e(t-1)\deg \phi+\deg \phi=et\deg \phi+(1-e)\deg \phi\leq et\deg \phi=\deg f,$ and hence by (\ref{dpeq6}) , we have \[w(g)=\bar{v}(g(\theta))>0,\] which contradicts (\ref{dpeq7}). Thus $\left(\frac{\phi(\theta)^e}{h(\alpha)}\right)^*$ is algebraic over $k_\alpha$ of degree $t.$ \\ 
		
	Conversely, assume that statements (i)-(iv) hold.
	Set $\delta=\bar{v}(\theta-\alpha)$, from (i) we have $(\alpha,\delta)$ is a $(K,v)$-minimal pair. Define a valuation $w=w_{\alpha,\delta},$ then $\phi$ is a minimal degree key polynomial for $w.$ To prove that $(\theta,\alpha)$ is a distinguished pair, by Theorem \ref{3.9}, it is enough to show that $f \in \operatorname{KP}(w).$ \\
	Let $\phi=\prod_{i=1}^n(x-\alpha_i),$ where $n=\deg\phi$. Then
	\[w(\phi)=\sum_{i=1}^n\overline{w}_{\alpha,\delta}(x-\alpha_i)=\sum_{i=1}^n\min\{\delta,\bar{v}(\alpha-\alpha_i)\}.\] Using assumption (i), it can be easily shown that $\min\{\delta,\bar{v}(\alpha-\alpha_i)\}=\bar{v}(\theta-\alpha_i),$ for all $i, \ 1\leq i\leq n.$ Therefore, we have
	\begin{align}\label{dpeq8}
	w(\phi)=\sum_{i=1}^n\bar{v}(\theta-\alpha_i)=\bar{v}(\prod_{i=1}^n(\theta-\alpha_i))=\bar{v}(\phi(\theta)).
	\end{align} By assumption (iii), $e$ is the smallest positive integer such that $e\bar{v}(\phi(\theta))=\bar{v}(h(\alpha))\in \Gamma_\alpha,$ i.e., $(\Gamma_\theta:\Gamma_\alpha)=e.$ Now suppose that $[k_\theta:k_\alpha]=t.$ By assumption (iv), we have $d(K(\theta)|K)=d(K(\alpha)|K).$ Keeping in mind that $K(\theta)|K$ and $K(\alpha)|K$ are both unibranched, we have
	\begin{align*}
	\frac{[K(\theta):K]}{[K(\alpha):K]}=(\Gamma_\theta:\Gamma_\alpha)[k_\theta:k_\alpha]=et,
		\end{align*}
		which implies that $\deg f=et\deg \phi.$ Write $f=\prod_{i=1}^{etn}(x-\theta_i),$  then
	\[w(f)=\sum_{i=1}^{etn}\overline{w}_{\alpha,\delta}(x-\theta_i)=\sum_{i=1}^{etn}\min\{\delta,\bar{v}(\alpha-\theta_i)\}.\] As $f=f^h,$ for any $\theta_i\in Z(f),$ there exist $\sigma\in \text{Aut}(\overline{K}|K^h)$ such that $\theta_i=\sigma(\theta),$ and hence, $\bar{v}(\theta_i-\alpha)=\bar{v}(\sigma(\theta)-\alpha)=\bar{v}\circ\sigma(\theta-\sigma^{-1}(\alpha))=\bar{v}(\theta-\sigma^{-1}(\alpha))\leq\delta,$ where the last inequality follows from assumption (i). Therefore, we have $\min\{\delta,\bar{v}(\alpha-\theta_i)\}=\bar{v}(\alpha-\theta_i),$ for all $i.$ Hence,
		\[w(f)=\sum_{i=1}^{etn}\bar{v}(\alpha-\theta_i)=\bar{v}(\prod_{i=1}^{etn}(\alpha-\theta_i))=\bar{v}(f(\alpha)).\]Let $f=\sum_{i= 0}^{et}f_i\phi^i$ be the $\phi$-expansion of $f,$ then $f_{et}=1,$ and by the above equation, we have $w(f)=\bar{v}(f(\alpha))=\bar{v}(f_0(\alpha))=w(f_0).$ By Lemma \ref{3.13}, $et\bar{v}(\phi(\theta))=\bar{v}(f(\alpha)),$ which implies  that $w(f)=etw(\phi)$ (see (\ref{dpeq8})). Let $i\in S_{w,\phi}(f)$ be arbitrary, then $w(f_i\phi^i)=w(f)=w(f_0),$ i.e., $iw(\phi)=w(f_0)-w(f_i)\in \Gamma_\alpha.$ But $e$ is the smallest positive integer such that $ew(\phi)\in \Gamma_\alpha,$ so $S_{w,\phi}(f)=\{0, ej_1, ej_2, \ldots, ej_s, et\},$ for some $0<j_1<j_2<\cdots<j_s<t.$ As $tw(h)=tew(\phi)=w(f)$, therefore
		\[\left(\frac{f}{h^t}\right)^*=\left(\frac{f_0}{h^t}\right)^*+\left(\frac{f_{ej_1}}{h^{t-j_1}}\right)^*\left(\left(\frac{\phi^e}{h}\right)^{j_1}\right)^*+\cdots+\left(\frac{f_{ej_s}}{h^{t-j_s}}\right)^*\left(\left(\frac{\phi^e}{h}\right)^{j_s}\right)^*+\left(\left(\frac{\phi^e}{h}\right)^t\right)^*.\] Since $\deg h,\deg f_i<\deg \phi$, by Lemma \ref{dpl12}, we have $h^*=h(\alpha)^*,$ and $f_i^*=f_i(\alpha)^*.$ Therefore,  \[\left(\frac{f(x)}{h(\alpha)^t}\right)^*=U\left(\left(\frac{\phi(x)^e}{h(\alpha)}\right)^*\right),\] where \[U(Y)=\left(\frac{f_0(\alpha)}{h(\alpha)^t}\right)^*+\left(\frac{f_{ej_1}(\alpha)}{h(\alpha)^{t-j_1}}\right)^*Y^{j_1}+\cdots+\left(\frac{f_{ej_s}(\alpha)}{h(\alpha)^{t-j_s}}\right)^*Y^{j_s}+Y^t\in k_\alpha [Y],\] is a monic polynomial over $k_\alpha$ of degree $t,$ whose constant term is nonzero. Also \[U\left(\left(\frac{\phi(\theta)^e}{h(\alpha)}\right)^*\right)=\left(\frac{f(\theta)}{h(\alpha)^t}\right)^*=0.\] By assumption (ii) it follows that $U(Y)$ is the minimal polynomial of $\left(\frac{\phi(\theta)^e}{h(\alpha)}\right)^*$ over $k_\alpha.$ Therefore, we have that $f$ is a lifting of a monic irreducible polynomial $U(Y),$ in an indeterminate $Y= \left(\frac{\phi(x)^e}{h(\alpha)}\right)^*$ over $k_\alpha,$ having degree $t\geq 1,$ with respect to $(\alpha,\delta)$ and $h.$ Hence, by Theorem \ref{ppth} we have, $f\in\operatorname{KP}(w).$
	\end{proof}
	
	\begin{remark}
		In the forward part of the above theorem, only $K(\alpha)|K$ is unibranched is required. Since $f,\phi\in \operatorname{KP}(w),$ by Proposition \ref{bp2}, $K(\alpha)|K$ is unibranched implies that $K(\theta)|K$ is also unibranched.
	\end{remark}
	
	As an application of above result, we have the following two theorems.
	
	\begin{theorem}\label{dpmt2}
		Let $(\theta,\alpha),$ and $(\theta,\beta)$ be two $(K,v)$-distinguished pairs such that $K(\theta)|K$ is unibranched, and $\phi, \psi$ be the minimal polynomials of $\alpha, \beta$ over $K$  respectively. Then
		\begin{enumerate}[(i)]
			\item $\Gamma_\alpha=\Gamma_\beta$,
			\item $k_\alpha=k_\beta$ ,
			\item $\bar{v}(\phi(\theta))=\bar{v}(\psi(\theta))$, and
			\item $d(K(\alpha)|K)=d(K(\beta)|K).$
		\end{enumerate}
	\end{theorem}
	\begin{proof}
		By definition of a distinguished pair, we have \[\bar{v}(\theta-\alpha)=\bar{v}(\theta-\beta), \ \ \text{and}\ \ \deg \alpha=\deg \beta.\] Let $\delta=\bar{v}(\theta-\alpha),$ then $(\alpha,\delta),$ and $(\beta,\delta)$ are both $(K,v)$-minimal pairs. We now define $w=w_{\alpha,\delta}.$ Since $\bar{v}(\alpha-\beta)\geq \delta,$ therefore, $(\beta,\delta)$ is also a minimal pair of definition for $w.$ Hence, $\phi$ and $\psi$ are both minimal degree key polynomials for $w.$ Let $f$ be the minimal polynomial of $\theta$ over $K,$ then by Theorem \ref{3.9}, $f\in\operatorname{KP}(w).$ As  $K(\theta)|K$ is unibranched, Proposition \ref{bp2} implies that both $K(\alpha)|K$ and $K(\beta)|K$ are also unibranched. Again as $\bar{v}(\alpha-\beta)\geq  \delta,$ so by Lemma \ref{3.12}, for any polynomial $g\in K[x]$ with $\deg g<\deg \phi=\deg \psi$, $\bar{v}(g(\beta)-g(\alpha))>\bar{v}(g(\alpha)).$ Therefore, we have \[\Gamma_\alpha = \Gamma_\beta, \ \text {and} \ k_\alpha= k_\beta,\] proving assertions (i) and (ii).
		As $\deg \phi=\deg \psi < \deg \theta,$ $w(\phi)=\bar{v}(\phi(\theta)),$ and $w(\psi)=\bar{v}(\psi(\theta)).$ Since $\phi, \psi$ are both same degree key polynomials for $w,$ by \cite[Theorem 3.9]{EN4} we have $w(\phi)=w(\psi)$.  Hence, assertion (iii) follows. Assertion (iv) directly follows from Theorem \ref{dpmt}.
		\end{proof}

		\begin{theorem}\label{dpth3}
			Let $\theta_i$ be elements in the algebraic  closure $\overline{K}$ of a valued field $(K,v)$ with respective minimal polynomials $\phi_i$ over $K,$ for $0\leq i\leq n.$ If $\theta=\theta_n, \theta_{n-1},\ldots, \theta_0$ is a saturated distinguished chain for $\theta,$ then the following conditions are satisfied.
			\begin{enumerate}[(i)]
				\item $\Gamma_{\theta_i}=\Gamma_{\theta_{i-1}}+\mathbb{Z}\bar{v}(\phi_{i-1}(\theta_i))$ and $\Gamma_{\theta_0}=\Gamma_v,$
				\item $k_{\theta_i}=k_{\theta_{i-1}}\left(\left(\frac{\phi_{i-1}(\theta_i)^{e_i}}{h_{i-1}(\theta_{i-1})}\right)^*\right),$ where $e_i$ is the smallest positive integer such that $e_i\bar{v}(\phi_{i-1}(\theta_i))=\bar{v}(h_{i-1}(\theta_{i-1}))\in \Gamma_{\theta_{i-1}},$ for some $h_{i-1}\in K[x]$ with $\deg h_{i-1}<\deg \phi_{i-1}$ and $k_{\theta_0}=k_v,$
				\item $[K(\theta_i):K]=(\Gamma_{\theta_i}:\Gamma_v)[k_{\theta_i}:k_v].$
			\end{enumerate}
		\end{theorem}
		\begin{proof}
			As $\theta=\theta_n, \theta_{n-1},\ldots, \theta_0$ is a saturated distinguished chain for $\theta,$ we have $(\theta_i,\theta_{i-1})$ is a distinguished pair for every $i,\ 0\leq i\leq n$ and  $\theta_0\in K.$  As $\phi_0\in K[x]$ is of degree one, it follows that $\phi_0=\phi_0^h.$ Since $(\theta_1,\theta_0)$ is a distinguished pair, therefore by Theorem \ref{3.9}, $\phi_1, \phi_0$ are both key polynomials for some valuation on $K[x],$ and hence by Proposition \ref{bp3}, $\phi_0=\phi_0^h$ implies that $\phi_1=\phi_1^h.$ Consequently, $\phi_i=\phi_i^h,\ \forall \ 0\leq i \leq n.$ So by Proposition \ref{bp1}, we have $K(\theta_i)|K$ is unibranched for every $i,\ 0\leq i\leq n.$  Since $(\theta_i,\theta_{i-1})$ is a distinguished pair for every $i,\ 0\leq i\leq n$, by using Theorem \ref{dpmt} (ii) and (iii), we have $\Gamma_{\theta_i}=\Gamma_{\theta_{i-1}}+\mathbb{Z}\bar{v}(\phi_{i-1}(\theta_i))$ and 
			$k_{\theta_i}=k_{\theta_{i-1}}\left(\left(\frac{\phi_{i-1}(\theta_i)^{e_i}}{h_{i-1}(\theta_{i-1})}\right)^*\right),$ where $e_i$ is the smallest positive integer such that $e_i\bar{v}(\phi_{i-1}(\theta_i))=\bar{v}(h_{i-1}(\theta_{i-1}))\in \Gamma_{\theta_{i-1}},$ for some $h_{i-1}\in K[x]$ with $\deg h_{i-1}<\deg \phi_{i-1}$. Also $\theta_0\in K,$ implies that $\Gamma_{\theta_0}=\Gamma_v$ and $k_{\theta_0}=k_v.$ Hence, assertions (i) and (ii) follows. Again as $\theta_0\in K,$ we get $K(\theta_0)|K$ is defectless. Therefore by repeated application of Theorem \ref{dpmt} (iv), we have $K(\theta_i)|K$ is defectless for every $i,\ 1\leq i\leq n.$ Hence, $K(\theta_i)|K$ is unibranched and defectless for every $i,\ 0\leq i\leq n,$ and assertion (iii) follows.
		\end{proof}
	
	\begin{remark}
		Theorem \ref{dpmt} and \ref{dpmt2} generalize Theorem 1.1 and 1.3 of \cite{AK2} respectively, and Theorem \ref{dpth3} generlaizes Theorem 1.1 of \cite{AN2}. Observe that using these results, one can easily prove that Theorem 1.4 and 1.5 of \cite{AK2} and Theorem 1.2 of \cite{AN2} also hold for arbitrary valued fields.
		\end{remark}
		
		\section{Abstract key polynomials}
	  Let $w$ be a valuation on $K[x].$ In this section, we first recall some properties of ABKPs and later generalize the results of \cite{SA3}, proved for Henselian valued fields, to arbitrary valued fields. Finally, we give the proofs of Theorem \ref{th1} and Conjecture \ref{1.2}.
	
	\begin{definition}
		A monic polynomial $Q$ in $K[x]$ is said to be an {\bf abstract key polynomial}  (abbreviated  ABKP) for $w$ if for each polynomial $f$ in $K[x]$  with $\deg f< \deg Q$ we have $\delta(f)<\delta(Q).$
	\end{definition}
	It is immediate from the definition that all monic linear polynomials are ABKPs for $w.$  Also an ABKP for $w$ is an irreducible polynomial. If $w$ has nontrivial support $\phi K[x],$ then $\phi$ is a highest degree ABKP for $w$ as $\delta(\phi)=\infty.$
	\begin{definition}
		Let $w$ be a valuation on $K[x].$ Then
		for a polynomial $Q$ in $K[x]$ the {\bf $Q$-truncation} of $w$ is a map $w_Q:K[x]\longrightarrow \Gamma_w$ defined by 
		$$ w_Q(f):= \min_{i\geq 0}\{w(f_iQ^i)\},$$
		where $\sum_{i\geq 0} f_i Q^i,$ $\deg f_i <\deg Q,$  is the $Q $-expansion of $f.$ 
	\end{definition}
	The $Q$-truncation  $w_Q$  of $w$ need not be a valuation \cite[Example 2.5]{NS}. However,  if $Q$ is an ABKP for $w,$ then $w_Q$ is a valuation on $K[x]$ (see \cite[Proposition 2.6]{NS}). If $Q\in \text{supp}(w),$ then $w_Q=w,$ otherwise, $w_Q$ is a valuation-transcendental extension and $Q$ is a minimal degree key polynomial for $w_Q$ (see \cite[Theorem 2.21]{Ma}). In either case, $Q$ is also an ABKP for $w_Q$ (By \cite[Corollary 2.22]{Ma}).
	
%

	In the following result, we recall  some basic properties of ABKPs  for  $w$   (see Proposition 2.10 and Lemma 2.11 of  \cite{NS},  Proposition 3.8, Corollaries 3.10, 3.11,  3.13 and Theorem 6.1 of  \cite{JN1}).
	\begin{proposition}\label{abmr}
		For  ABKPs, $Q$ and $Q'$ for $w$ the following holds:
		\begin{enumerate}[(i)]
			\item If  $\delta(Q)<\delta(Q'),$ then $w_Q(Q')<w(Q').$ 
			\item If $\deg Q = \deg Q',$ then $$w(Q)<w(Q')\iff w_Q(Q')<w(Q')\iff \delta(Q)<\delta(Q').$$ Hence $Q'\in\Phi(w_Q,w)$ in this case.
			\item  Suppose that  $\delta(Q)<\delta(Q').$ For any polynomial $f\in K[x],$ we have
			\begin{align*}
				w_Q(f)&\leq w_{Q'}(f)\leq w(f),\\
				w_{Q}(f)&=w(f)\implies w_{Q'}(f)=w(f),~\text{and}\\
				w_{Q'}(f)&<w(f)\implies w_Q(f)<w_{Q'}(f).
			\end{align*}
			\item If $Q'\in \Phi(w_Q,w),$ then $Q$ and $Q'$ are key polynomials for $w_Q.$ Moreover, $w_{Q'}=[w_Q; Q',w(Q')].$
			\item 	 Every $F\in\Phi(w_Q,w)$ is an ABKP for $w$ and $\delta(Q)<\delta(F).$
			
		\end{enumerate}
	\end{proposition}
	
	To obtain a relation between ABKPs and distinguished pairs, we need the following lemma. 
	\begin{lemma}\label{abl1}
		Let $(K,v)$ be a valued field, $w$ be an extension of $v$ to $K[x],$ and $\overline{w}$ be any common extension of $w$ and $\bar{v}$ to $\overline{K}[x].$ Let $F\in K[x]$ be an ABKP for $w$ and $Q\in K[x]$ be any polynomial. If $(F,Q)$ is a distinguished pair and $F$ is irreducible over $K^h,$ then the following hold:
		\begin{enumerate}[(i)]
			\item If $\theta$ and $\alpha$ are optimizing roots of $F$ and $Q$ respectively with respect to $\overline{w}$, then $(\theta,\alpha)$ is a $(K,v)$-distinguished pair.
			\item The polynomial $Q$ is an ABKP for $w.$
		\end{enumerate}
		\end{lemma}
		\begin{proof}
		\noindent (i)~ Since $(F,Q)$ is a distinguished pair, there exist a root $\theta_0$ of $F$ and a root $\alpha_0$ of $Q$ such that $(\theta_0,\alpha_0)$ is a $(K,v)$-distinguished pair. Therefore, by Theorem \ref{3.9}, $F$ and $Q$ are key polynomials for some valuation on $K[x].$ As  $F$ is irreducible over $K^h$  i.e. $F=F^h,$ Proposition \ref{bp3} (ii) implies that $Q=Q^h,$ i.e., every $K$-conjugate of $\theta$ and $\alpha$ are also $K^h$-conjugates of $\theta$ and $\alpha$ respectively. So to prove that $(\theta,\alpha)$ is a  $(K,v)$-distinguished pair,  it is enough to show that $\theta$ is a root of $F$ for which \[\bar{v}(\theta-\alpha)=\max\{\bar{v}(\theta'-\alpha)\mid \theta' ~\text{a $K$-conjugate of $\theta$}\}.\]  Since $F$ is  an ABKP for $w$ and $\deg Q<\deg F,$ so $\delta(Q)<\delta(F)$ which  implies that 
		\begin{align}\label{e4.1}
		\overline{w}(x-\alpha)=\delta(Q)<\delta(F)=\overline{w}(x-\theta).	
		\end{align}
		On applying strong triangle law to the  above inequality,  $\bar{v}(\theta-\alpha)=\overline{w}(x-\alpha)=\delta(Q).$  As $\alpha$ is an optimizing root of $Q,$ so
		$$\overline{w}(x-\alpha')\leq \overline{w}(x-\alpha)=\bar{v}(\theta-\alpha)<\overline{w}(x-\theta),$$
		for each $K$-conjugate $\alpha'$ of $\alpha,$ which  implies that
		\begin{align}\label{e4.2}
			\bar{v}(\theta-\alpha')=\overline{w}(x-\alpha')\leq\overline{w}(x-\alpha)=\bar{v}(\theta-\alpha).
		\end{align}
		As $F=F^h$, so for any $K$-conjugate $\theta'$ of $\theta$ there exists some $K$-conjugate $\alpha'$ of $\alpha$ such that $\bar{v}(\theta'-\alpha)=\bar{v}(\theta-\alpha').$ Therefore from  (\ref{e4.2}) we have that 
		\begin{align*}
			\bar{v}(\theta'-\alpha)=\bar{v}(\theta-\alpha')\leq\bar{v}(\theta-\alpha)
		\end{align*}
		for every $K$-conjugate $\theta'$ of $\theta.$
		
		\noindent (ii)~ Let $g\in K[x]$ be any polynomial with $\deg g<\deg Q,$ we need to show that $\delta(g)<\delta(Q).$ By (i) $(\theta,\alpha)$ is a $(K,v)$-distinguished pair, so for any optimizing root $\beta$ of $g,$  $\deg\beta<\deg\alpha,$  implies
		$\bar{v}(\theta-\beta)<\bar{v}(\theta-\alpha)=\delta(Q),$ which together with (\ref{e4.1}) shows that $\delta(g)<\delta(Q).$
		\end{proof}

	\begin{theorem}\label{abth}
		Let $(K,v)$ be a valued field and $w$ an extension of $v$ to $K[x].$ Then for any ABKP $F$ for $w$ which is irreducible over $K^h$ and $Q\in K[x],$ the following are equivalent:
		\begin{enumerate}[(i)] 
			\item  $(F,Q)$ is a distinguished pair.
			\item  $Q$ is an ABKP for $w$ such that  $w_Q<w,$ $F\in\Phi(w_Q,w)$  and  $\deg F>\deg Q.$
			\item   $Q$ is an ABKP for $w$ such that  $w_Q<w$ and  $F$ is a   nontrivial lifting of some monic irreducible polynomial $G(Y)\neq Y$ in $k_{w_Q}$ with respect to $w_Q.$   
		\end{enumerate}
	\end{theorem}
	\begin{proof}
		(i)$\implies$(ii)
		Since $F$ is an ABKP for $w$ and $(F,Q)$ is a distinguished pair, so by Lemma \ref{abl1} (ii) and  Proposition \ref{abmr} (i),  $Q$ is an ABKP for $w$ and   $w_Q(F)<w(F).$
		Assume  to the contrary that $F\notin \Phi(w_Q,w).$ Let  $g$ in $K[X]$ be such that $g\in\Phi(w_Q,w).$ Then by  Proposition \ref{abmr} (v),   $g$ is an ABKP for $w$ and $\delta(Q)<\delta(g).$  Let $\overline{w}$ be any common extension of $w$ and $\bar{v}$ to $\overline K[x].$ If $\alpha$ and $\beta$ are optimizing roots of $Q$ and $g$ respectively, then
		$$\overline{w}(x-\alpha)=\delta(Q)<\delta(g)=\overline{w}(x-\beta)$$ which implies
		\begin{align}\label{eq3.18}
			\bar{v}(\beta-\alpha)=\overline{w}(x-\alpha)<\overline{w}(x-\beta) .
		\end{align}
	As $\deg g<\deg F$ and $F$ is an ABKP for $w,$ so we have
		$$\overline{w}(x-\beta)=\delta(g)<\delta(F)=\overline{w}(x-\theta),$$
		for some optimizing root $\theta$   of $F,$  and hence 
		$$\bar{v}(\theta-\beta)=\overline{w}(x-\beta)<\overline{w}(x-\theta).$$
		The above inequality together with (\ref{eq3.18}), on applying strong triangle law gives
		\begin{align}\label{eq3.20}
			\bar{v}(\theta-\alpha)=\bar{v}(\beta-\alpha)<\bar{v}(\theta-\beta).
		\end{align}
		As $(F,Q)$ is a distinguished pair, so by Lemma \ref{abl1} (i),  $(\theta,\alpha)$  is a $(K,v)$-distinguished pair and hence, 
		$\deg\beta<\deg\theta$ implies that $\bar{v}(\theta-\beta)\leq\bar{v}(\theta-\alpha)$ which contradicts  (\ref{eq3.20}). \\
		(ii)$\implies$(iii) As $F\in\Phi(w_Q,w) \subset \operatorname{KP}(w_Q),$ $F$ is a key polynomial for $w_Q,$ therefore (iii) follows immediately from Theorem \ref{ppth}.\\
			(iii)$\implies$(i) As $Q$ is a minimal degree key polynomial for $w_Q,$ and $F\in\operatorname{KP}(w_Q)$ (by Theorem \ref{ppth}), therefore in view of Theorem \ref{3.9}, we have $(F,Q)$ is a distinguished pair.
		\end{proof}
	
	Note that the above two results are already proved in \cite[Lemma 2.3]{SA3} and \cite[Theorem 1.13]{SA3} for Henselian valued fields and valuations on $K(x).$
	\begin{definition}
		A family $\Lambda=\{Q_i\}_{i\in\Delta}$ of ABKPs for $w,$ indexed by a well-ordered set $\Delta,$ is said to be a {\bf complete sequence of ABKPs} for $w$ if the following conditions are satisfied:
		\begin{enumerate}[(i)]
			\item $\delta(Q_i)\neq \delta(Q_j)$ for every $i\neq j\in\Delta.$
			\item $\Lambda$ is well-ordered with respect to the ordering given by $Q_i< Q_j$ if and only if  $\delta(Q_i)<\delta(Q_j)$ for every $i<j\in \Delta.$ 
			\item For any $f\in K[x],$ there exists some $Q_i \in \Lambda$ such that $\deg Q_i\leq \deg f$ and  $w_{Q_i}(f)=w(f).$
		\end{enumerate}
	\end{definition}
	
	Every valuation $w$ on $K[x]$ admits a complete sequence of ABKPs for $w.$ 	All  extensions of $v$ to $K[x]$ can be classified using the  notion of complete sequence of abstract key polynomials  (see \cite{W}). 
	
	\begin{remark} 
		As shown in \cite[Remark 4.6]{W}, there is a complete sequence  $\Lambda=\{Q_i\}_{i\in\Delta}$ of ABKPs  for  $w$ having the following properties:
		\begin{enumerate}[(i)]
			\item $\Delta=\bigcup_{j\in I}\Delta_j$ with $I=\{0,1,\ldots, N\}$ or $\mathbb{N}\cup\{0\},$ and for each $j\in I$ we have $\Delta_j=\{j\}\cup\vartheta_{j},$ where $\vartheta_j$ is an ordered set without a last element or is empty.
			\item $Q_0$ is a monic polynomial of degree one.
			\item For all $j\in I\setminus \{0\}$ and $i \in\vartheta_{j-1},$ we have $j-1<i<j.$ 
			\item All  polynomials $Q_i$ with $i\in\Delta_j$ have the same degree and  have degree strictly  less than  the degree of the polynomials $Q_{i'}$ for every $i'\in\Delta_{j+1}.$
			\item For each $i<i'\in\Delta$ we have $w(Q_i)<w(Q_{i'}).$
		\end{enumerate}
	\end{remark}
	
	Even though the complete sequence $\{Q_i\}_{i\in\Delta}$  of ABKPs  for $w$  is not unique, the cardinality of $I$ and the degree of an abstract key polynomial $Q_i$ for each $i\in I$ are uniquely determined by $w.$
	We say that the ordered set $\Delta$ has a last element if and only if  
	\begin{align}
		I&=\{0,1,\ldots,N\}~\text { and  
			$\Delta_N=\{N\},$ i.e., $\vartheta_N=\emptyset.$}
	\end{align}
	The connection between a complete sequence of ABKPs and MLV chain for any valuation-transcendental and valuation-algebraic extension $w$ on $K(x)$ is given in \cite{SA1} and \cite{SA2}, respectively. These results show that for any given complete sequence of ABKPs, we can explicitly construct an MLV chain of $w$, and conversely.
	\begin{remark}\label{ntsr}
		With Proposition \ref{abmr} in mind, it is easy to  show that Theorems 3.1 and 3.2 of \cite{SA1} also hold for valuations with nontrivial support. Moreover, if $\text{supp}(w)=\phi K[x],$ then $\phi$ is the last element of the complete sequence of ABKP for $w.$
	\end{remark}
	
	 In \cite{SA3}, Theorem 1.21 gives the existence of a saturated distinguished chain using a complete sequence of ABKPs for Henselian valued fields. Using Lemma \ref{abl1}, Theorem \ref{abth}, and Remark \ref{ntsr}, it immediately follows that this theorem extends to arbitrary valued fields and valuations $w$ with nontrivial support. More precisely,
	
	 \begin{theorem}\label{3.1.22}
	 	Let $(K,v)$ be a valued field
	 	and  $w$  an extension of $v$ to $K[x].$ Assume that for a complete sequence $\{Q_i\}_{i\in\Delta}$ of ABKPs for $w,$    $\vartheta_j=\emptyset$  for every $j\in I.$ If $I\neq \{0\},$ then for each $n\in I\setminus \{0\},$  $Q_n$ has a saturated distinguished chain.
	 \end{theorem}
	 \begin{corollary}\label{cor3.1.7}
	 	Let $w=w_Q,$ for some ABKP $Q$ for $w.$ If $(Q=Q_n,\ldots, Q_0)$ is a saturated distinguished chain for $Q,$ then $\{Q_0, Q_{1},\ldots,Q_n=Q\}$ is a complete sequence of ABKPs for $w.$ 
	 \end{corollary}
	 \begin{proof}[Proof of Theorem \ref{th1}]
	 	(i)$\iff $(ii) follows from Theorems 3.1 and 3.2 of \cite{SA1} and Remark \ref{ntsr}.\\ (ii)$\iff$(iii) follows from Theorem \ref{3.1.22} and Corollary \ref{cor3.1.7} as $w_n=w_{\phi_n}$.\\ (iii)$\implies$(iv) As $\phi_0\in K[x]$ is of degree one, it follows that $\phi_0=\phi_0^h.$ Since $(\phi_1,\phi_0)$ is a distinguished pair, therefore by Theorem \ref{3.9}, $\phi_1, \phi_0$ are both key polynomials for some valuation on $K[x],$ and hence by Proposition \ref{bp3} (ii), $\phi_0=\phi_0^h$ implies that $\phi_1=\phi_1^h.$ Consequently, we have $\phi_i=\phi_i^h,\ \forall \ 0\leq i \leq n.$  So (iv) follows immediately by repeated application of Lemma \ref{abl1} (i).\\
	 	(iv)$\implies$(v) Arguing as above, we have  $K(\theta_i)|K$ is unibranched for every $i,\ 0\leq i\leq n.$ Also as $\theta_0\in K,$ we get $K(\theta_0)|K$ is defectless. Therefore by repeated application of Theorem \ref{dpmt} (iv), we have $K(\theta_i)|K$ is defectless for every $i,\ 1\leq i\leq n.$ \\
	 	(v)$\implies$(i) Since either $\phi_n$ is a minimal degree key polynomial for $w_n$ or $\phi_n\in \text{supp}(w_n),$ therefore by Theorem \ref{T2} and Corollary \ref{2.11}, we have \[d(w_n|v)=d(K(\theta_n)|K)=1.\]  Hence, using Theorem \ref{vth1}, (i) follows.
	 	\end{proof}
	
	In view of Theorem \ref{th1}, we have that if $K(\theta)|K$ is unibranched and defectless, then the $\text{depth}(\theta)$ is equal to the number of terms in the distinguished chain of $\theta$ over $K.$ Hence, as an immediate consequence of Theorem \ref{dpth3}, we obtain the following result.
	
	\begin{corollary}
		Let $K(\theta)|K$ be a finite simple unibranched defectless extension, and $s, t$ be the minimal number of generators of $k_\theta|k_v$ and $\Gamma_\theta|\Gamma_v$ respectively. Then, \[\max\{s,t\}\leq \textup{depth}(\theta)\leq s+t.\]
	\end{corollary}
	
	The depth of a finite simple extension $L|K$ is defined as 
	\[\text{depth}(L|K,v)=\min\{\text{depth}(\theta)\mid L=K(\theta)\}.\] Note that different generators of the same extension may have different depths, but the residue field $k_L$ and the value group $\Gamma_L$ of the extension is independent of choice of the generator. Therefore, the above corollary holds for every generator of the simple extension and hence, we have
	
	\begin{corollary}\label{lc}
		Let $L|K$ be a finite simple unibranched defectless extension, and $s, t$ be the minimal number of generators of $k_L|k_v$ and $\Gamma_L|\Gamma_v$ respectively. Then, \[\max\{s,t\}\leq \textup{depth}(L|K,v)\leq s+t.\]
	\end{corollary}
	
	\begin{remark}
		The above corollary is the generalization of Conjecture \ref{1.2}, and as a consequence, it also generalizes the results of Section 3 of \cite{NN}.
	\end{remark}
	
	
	\section*{Acknowledgement}
	The research of the first author is supported by the UGC
	(Reference no.\ 231620082501).
	

\end{document}